\documentclass{amsart}[11pt]

\usepackage{amsmath}
\usepackage{amssymb}
\usepackage{amsthm}
\usepackage[mathscr]{eucal}
\usepackage{mathrsfs}
\usepackage{psfrag}
\usepackage{fullpage}
\usepackage{color}
\usepackage{eucal}
\usepackage[dvips]{graphicx}
\usepackage{amsfonts}%
\usepackage{amsaddr}
\usepackage{graphicx}
\usepackage{graphics}
\usepackage{enumerate}
\usepackage{dsfont}

\usepackage[margin=1in]{geometry}
\usepackage{float}
\graphicspath {{/}}

\newtheorem{lemma}{Lemma}
\newtheorem{theorem}{Theorem}

\newtheorem{condition}{Condition}

\newtheorem{remark}{Remark}

\newcommand{\xe}{X^\varepsilon}
\newcommand{\ye}{Y^\eta}
\newcommand{\yeps}{Y^\varepsilon}

\author{S. Bourguin, S. Gailus, and K. Spiliopoulos}
\address{Boston University, Department of Mathematics and Statistics\\ 111 Cummington Mall, Boston, MA 02215, USA}
\email[Solesne Bourguin]{bourguin@math.bu.edu}
\email[Siragan Gailus]{siragan@math.bu.edu}
\email[Konstantinos Spiliopoulos]{kspiliop@math.bu.edu}
\thanks{KS was partially supported by the National Science Foundation (DMS 1550918)}
\title[Analysis of slow-fast systems driven by fBm]{Typical dynamics and fluctuation analysis of slow-fast systems driven by fractional Brownian Motion}

\begin{document}

\begin{abstract}
This article studies typical dynamics and fluctuations for a slow-fast dynamical system perturbed by a small fractional Brownian noise. Based on an ergodic theorem with explicit rates of convergence, which may be of independent interest, we characterize the asymptotic dynamics of the slow component to two orders (i.e., the typical dynamics and the fluctuations). The limiting distribution of the fluctuations turns out to depend upon the manner in which the small-noise parameter is taken to zero relative to the scale-separation parameter. We study also an extension of the original model in which the relationship between the two small parameters leads to a qualitative difference in limiting behavior. The results of this paper provide an approximation, to two orders, to dynamical systems perturbed by small fractional Brownian noise and subject to multiscale effects.

\end{abstract}

\maketitle

\section{Introduction}
Dynamical systems exhibiting multiple characteristic scales in space or time arise naturally as models in a great variety of applied fields, including physics, chemistry, biology, neuroscience, meteorology, and mathematical finance, to name a few. It is moreover common to incorporate random perturbations into these models in order to account for imperfect information or to capture random phenomena. The particular case in which the perturbing noise is a standard Brownian motion has been studied extensively. With this choice, crucially, the Markov property and semimartingale structure of the standard Brownian motion are embedded in the system. While the analysis is simplified insofar as a host of well-developed theoretical tools may be brought to bear, important limits are placed on the flexibility of the model. For example, a physical dynamical system exhibiting long-range dependence or a particular sort of self-similarity may not be amenable to accurate description by a model driven by standard Brownian noise.

In this paper, we consider a model in which some of the random perturbation arises from a fractional Brownian motion (fBm), thereby making it possible to capture dynamical features that are out of the scope of the standard Brownian motion. More precisely, we consider $(\xe, \ye)_T = \{(\xe_t,
\ye_t)\}_{0\leq t\leq T}$ evolving in
$\mathcal{X}\times\mathcal{Y}:=\mathbb{R}^{m}\times\mathbb{R}^{d-m}$
according to the stochastic differential equation
\begin{equation}
\label{Eq:ModelSystem}
\begin{cases} d\xe_t =
c(\xe_t, \ye_t) dt + \sqrt\epsilon \sigma(\ye_t) dW^H_t \\ d\ye_t = \frac{1}{\eta} f(\ye_t) dt
+ \frac{1}{\sqrt\eta} \tau(\ye_t) dB_t \\
\xe_0=x_0\in\mathcal{X}, \hspace{1pc} \ye_0=y_0\in\mathcal{Y}.
\end{cases}
\end{equation}

Here, $W^H$ is a fractional Brownian motion with Hurst index
$H\in(1/2,1)$ and $B$ is a standard Brownian motion independent of $W^H$. The term $dW^H$ is to be understood in the sense of pathwise integration, although this pathwise integral coincides in our framework with the analogous divergence integral, and we shall freely and frequently interpret it as such in order to apply tools of Malliavin calculus (see Remark \ref{r:integrals} and Appendix \ref{A:Appendix} for a discussion of this point and for details on Malliavin calculus and integration with respect to fBm). $\varepsilon:=(\epsilon,\eta)\in\mathbb{R}^2_+$ is a pair of small positive parameters. Note that as $\varepsilon := (\epsilon, \eta)$ is taken to vanish, $\xe$ is the slow component and is perturbed by small noise, while $\ye$ is the fast component and feeds into the dynamics of $\xe$.

The main results of this work provide a rigorous description of the asymptotic behavior, to two orders, of $\xe$ as $\varepsilon:=(\epsilon, \eta)\to0$. We first show that $\xe$ converges in an $L^p$ sense, and at a particular rate, to a deterministic limiting process $\bar{X}$, which we interpret as the typical behavior of $X^\varepsilon$. We then derive a limit in distribution of the (appropriately-rescaled) fluctuations $\theta^\varepsilon:=\frac{1}{\sqrt{\epsilon}}(X^\varepsilon-\bar{X})$ about the limiting process. The limiting distribution of the fluctuations turns out to depend upon the manner in which the small asymptotic parameters are taken to vanish, even as the typical behavior does not exhibit any such dependence. In deriving the limit of the fluctuations, we assume for this reason that one is considering a class of pairs $\varepsilon := (\epsilon, \eta)$ for which there exists $\lambda\in[0, \infty)$ with $\lim_{\varepsilon\to0}\frac{\sqrt\eta}{\sqrt\epsilon} = \lambda$. For example, a functional dependence $\eta=\eta(\epsilon)$ such that $\lim_{\epsilon\to0}\frac{\sqrt\eta}{\sqrt\epsilon}=:\lambda\in[0,\infty)$ is more than sufficient.

The novelty of our setup lies in the nature of the small perturbing noise, which we take to be a fractional Brownian motion rather than a standard Brownian motion. Moreover, we allow the dynamics to evolve in the full Euclidean space, and apart from the diffusion coefficient in the fast component, we do not assume that coefficients are bounded or have bounded derivatives. Consequently, we rely in the proofs of our ergodic theorem, Theorem \ref{t:ergodic1}, and main results upon a-priori uniform bounds on both $\xe$ and its Malliavin derivative $D\xe$ with respect to the fractional Brownian noise $W^H$. The necessary bounds are derived in Lemmata \ref{l:sigmadw}, \ref{l:xbound}, and \ref{l:Dxbound}. In establishing the limit in distribution of the fluctuations, we also make use of recent results of \cite{hl:averagingdynamics}, which carry over to our setting.

Note that in (\ref{Eq:ModelSystem}), we have taken $\sigma$ to depend upon the fast variable only and not upon the slow. There are two reasons for this restriction, both relating to the fact that, by the independence of $W^H$ and $B$, $D\sigma(Y^\eta) \equiv 0$ whereas in general $D\sigma(X^\varepsilon,Y^\eta)$ would be nontrivial (recall that $D$ is the Malliavin derivative with respect to the fractional Brownian noise $W^H$). The first reason is technical in nature. As mentioned in the preceding paragraph, our proofs rely upon a-priori uniform bounds on $\xe$ and $D\xe$. To derive these bounds, we invoke a maximal inequality for the stochastic integral with respect to $W^H$, which in turn requires us to control the Malliavin derivative of the integrand. Thus, if $\sigma$ were allowed to depend upon the slow variable, we would encounter a closure problem in that to obtain a bound on the $k^{th}$-order Malliavin derivative $D^k\xe$ one would need first to obtain a bound on the $(k+1)^{st}$-order Malliavin derivative $D^{k+1}\xe$, and so on in a cascading fashion. In very special cases, e.g., in one dimension, it is possible to circumvent the problem, but this would seem to be the exception rather than the rule. The second reason has to do with modelling considerations. If one would like to interpret the slow component as the solution of an ODE perturbed by a small fractional Brownian noise, it is reasonable to formulate the model in such a way as for this noise to be centered, i.e., for the stochastic integral with respect to $W^H$ to have mean zero. In our setup, the pathwise integral coincides with the divergence integral, which is always centered (see Appendix \ref{A:Appendix} for details on Malliavin calculus and integration with respect to fBm). On the other hand, if one were to allow $\sigma$ to depend upon the slow variable then the pathwise integral would not typically be centered. It is worth noting that one's hands are tied here insofar as general results guaranteeing the existence of unique solutions of the system (\ref{Eq:ModelSystem}) are known only when the integral with respect to $W^H$ is interpreted in the pathwise sense.

In Section \ref{S:Extensions}, we study the typical behavior and fluctuations limit in the context of an extended model generalized from \eqref{Eq:ModelSystem}. The extended model takes the form
\begin{equation}
\label{Eq:ExtendedModel}
\begin{cases} d\xe_t &=
\frac{\sqrt{\epsilon}}{\sqrt{\eta}}b(\xe_t, \yeps_t)dt + c(\xe_t, \yeps_t) dt + \sqrt\epsilon \sigma(\yeps_t) dW^H_t \\ d\yeps_t &= \frac{1}{\eta} f(\yeps_t) dt + \frac{1}{\sqrt{\epsilon\eta}} g(\yeps_t) dt
+ \frac{1}{\sqrt\eta} \tau(\yeps_t) dB_t \\
\xe_0&=x_0\in\mathcal{X}, \hspace{1pc} \yeps_0=y_0\in\mathcal{Y}.
\end{cases}
\end{equation}
Recall that in the context of the original model, the typical behavior does not depend upon the manner in which the small asymptotic parameters are taken to vanish, and that the fluctuations analysis can be carried through assuming that one is concerned with a class of pairs $\varepsilon:=(\epsilon, \eta)$ for which there exists $\lambda\in[0,\infty)$ with $\lim_{\varepsilon\to0}\frac{\sqrt\eta}{\sqrt\epsilon}=\lambda$. Moving to the extended model, however, the introduction of the terms corresponding to the coefficients $b$ and $g$ introduces a qualitative discrepancy between regimes that is reflected even in the typical behavior. Accordingly, when we are considering the extended model, we not only assume, from the beginning, the existence of the limit $\lambda\in[0,\infty)$, but also distinguish two possibilities:
\begin{enumerate}[(i)]
	\item $\lambda = 0$, the `first regime' or `homogenization regime'
	\item $\lambda \in (0, \infty)$, the `second regime' or `averaging regime.'
\end{enumerate}
To obtain the limit in distribution of the fluctuations in the context of the extended model, we further assume that the convergence of $\frac{\sqrt\eta}{\sqrt\epsilon}$ to $\lambda$ takes place at a particular rate with respect to $\sqrt\epsilon$. Note that the homogenization regime is that in which the term $\frac{\sqrt{\epsilon}}{\sqrt{\eta}} b(\xe, \yeps)$ is asymptotically singular. In precise analogy to the analysis done in the case of perturbation by standard Brownian motion in \cite{gs:disctime, spiliopoulos2014fluctuation}, we shall see that the limiting contribution of the asymptotically-singular term can be captured in terms of the solution of an appropriate Poisson equation. The proofs of our results for the original model \eqref{Eq:ModelSystem} then carry over with minor modifications. Having already presented the full proofs for \eqref{Eq:ModelSystem}, we therefore describe only the adjustments necessary for \eqref{Eq:ExtendedModel}. The extended model \eqref{Eq:ExtendedModel} is particularly relevant when, for example, a fast intermediate scale forms part of the slow component. The scaling in front of the term corresponding to the coefficient $g$ is that which results in a nontrivial limiting contribution in the event that additional intermediate fast scales form part of the main fast component.

In the case of perturbation by standard Brownian motion, the
literature on similar limiting theorems for stochastic dynamical
systems is extensive. We mention here for completeness
\cite{BaierFreidlin,DupuisSpiliopoulos,Freidlin1978, FS,
  FWBook,gs:disctime,gs:statinf,Guillin,KlebanerLiptser,LiptserPaper,pardoux2001poisson,
  pardoux2003poisson,spiliopoulos2014fluctuation,Spiliopoulos2012}, which contain results on
related typical dynamics, central limit theorems, and large
deviations. The corresponding literature in the case of perturbation
by fractional Brownian motion is quite sparse. The most relevant
result in our case is the recent work \cite{hl:averagingdynamics}, which studies related typical behavior of systems similar to \eqref{Eq:ModelSystem}. Our results on typical behavior differ from those of \cite{hl:averagingdynamics} in that we allow most coefficients to grow polynomially in the fast variable. Consequently, as discussed above, we must derive certain a-priori bounds in order to establish our ergodic theorem, Theorem \ref{t:ergodic1}, to which we appeal in turn in establishing our main results. In this way, we obtain an explicit rate of convergence to the typical behavior. To complement the results on typical behavior found in \cite{hl:averagingdynamics} and in this work, we then derive a limit in distribution of the fluctuations, which characterizes the limiting behavior of the slow component to next order.

Let us also mention the very recent preprint \cite{PeiInahamaXu2020} (appearing on arXiv during the reviewing process of this article), in which the authors study an averaging principle (although neither homogenization nor fluctuations and not in the small-noise regime) for systems similar to (\ref{Eq:ModelSystem}), but with the drift coefficient in the slow component assumed to be uniformly bounded and the diffusion coefficient to depend on the slow component instead of the fast component. The uniform bound on the drift coefficient facilitates the derivation of a-priori bounds on the slow component and the fact that the diffusion term does not depend on the fast component means that the fBm term does not have to be averaged out (in contrast to our work). In addition, in \cite{PeiInahamaXu2020}, the coefficients of the fast component are allowed to depend upon both slow and fast components, whereas we focus in this work on the case in which the dependence is on the fast component and not the slow. The reason for this is essentially that we have chosen to obtain a-priori bounds on the slow component and its Malliavin derivative with respect to the driving fBm from independence of $\sigma(\ye)$ and $W^H$, without the need for additional assumptions (see also the discussion above equation (\ref{Eq:ExtendedModel})).

Let us now explain the organization of the rest of this paper. In Section
\ref{SS:MainResults}, we introduce notation, present our main assumptions, and state our main results. Section
\ref{S:TypicalBehavior} contains proofs of results related to
the typical behavior of $\xe$ as $\varepsilon\to0$, including supporting lemmata and our ergodic theorem. Section
\ref{S:Fluctuations} contains proofs of tightness and convergence in
distribution of the (appropriately-rescaled) fluctuations of $\xe$ about the limit
$\bar{X}$ as $\varepsilon\to 0$. Section \ref{S:Extensions}
extends our results from the original model \eqref{Eq:ModelSystem} to the extended model
\eqref{Eq:ExtendedModel}. Finally, for the convenience of the reader, Appendix
\ref{A:Appendix} collects those definitions and tools related to fractional Brownian motion, Malliavin calculus, and stochastic integration with respect to fractional Brownian motion, that are used in this paper.

\section{Notation, Conditions, and Main Results}\label{SS:MainResults}
In this section we introduce notation, present our main assumptions, and state our main results.

We will denote by $A:B$ the Frobenius inner product $\Sigma_{i,j}[a_{i,j}\cdot b_{i,j}]$ of matrices $A=(a_{i,j})$ and $B=(b_{i,j})$.
We will use single bars $|\cdot|$ to denote the Frobenius (or Euclidean) norm of a matrix, and double bars $||\cdot||$ to denote the operator norm.

Condition \ref{c:regularity} imposes conditions of growth and regularity on the drift and diffusion coefficients of the model.
\begin{condition}\label{c:regularity}\hspace{1pc}\newline
\hspace{1pc}\newline
\noindent Conditions on $c$:
\begin{enumerate}[-]
	\item $\exists$\hspace{0.5pc}$(K,q)\in\mathbb{R}^2_+$, $r\in[0,1)$; $\left|c(x,y)\right|\leq K(1+|x|^r)(1+|y|^q)$
	\item $\exists$\hspace{0.5pc}$(K,q)\in\mathbb{R}^2_+$; $\left|\nabla_x c(x,y)\right|+\left|\nabla_x\nabla_x c(x,y)\right|\leq K(1+|y|^q)$
	\item $c$, $\nabla_x c$, $\nabla_x \nabla_x c$, and $\nabla_y \nabla_y c$ are continuous in $(x,y)$
	\item $c$, $\nabla_x c$, and $\nabla_x \nabla_x c$ are locally H\"older continuous in $y$ uniformly in $x$
\end{enumerate}

\noindent Conditions on $\sigma$:
\begin{enumerate}[-]
	\item $\exists$\hspace{0.5pc}$(K,q)\in\mathbb{R}^2_+$; $|\sigma(y)|\leq K(1+|y|^q)$
	\item $\sigma\sigma^T$ is uniformly nondegenerate
\end{enumerate}

\noindent Conditions on $f$ and $\tau$:
\begin{enumerate}[-]
	\item $f$ and $\tau\tau^T$ are twice differentiable, and, along with their first and second derivatives, are locally H\"older continuous
	\item $\tau\tau^T$ is uniformly bounded and uniformly nondegenerate.
\end{enumerate}

\end{condition}
Condition \ref{c:recurrencebasic} is a basic condition of recurrence type on the fast component, yielding ergodic behavior.
\begin{condition}\label{c:recurrencebasic}\hspace{1pc}\newline
\begin{align*}
\lim_{|y|\to\infty} y \cdot f(y) &= -\infty.
\end{align*}
\end{condition}

To derive most of our results we shall in fact assume a stronger recurrence condition.
\begin{condition}\label{c:recurrence}\hspace{1pc}\newline
For real constants $\alpha>0$, $\beta\geq2$, and $\gamma>0$, we shall write:
\begin{enumerate}[-]
	\item Condition \ref{c:recurrence}-$(\alpha,\beta)$: one has
	\begin{align*}
	y \cdot f(y) + \alpha |y|^\beta + \frac12 (\beta-2+d-m)\sup_{\tilde y\in\mathcal{Y}}|\tau(\tilde y)|^2 \leq 0
	\end{align*}
	for $|y|$ sufficiently large
	\item Condition \ref{c:recurrence}-$(\alpha,\beta,\gamma)$:
	Condition \ref{c:recurrence}-$(\alpha,\beta)$ holds and, moreover, one has $||\nabla_x c(x,y)|| \leq \gamma|y|^\beta$ for $|y|$ sufficiently large.
\end{enumerate}
\end{condition}

\begin{remark} Clearly, Condition \ref{c:recurrencebasic} is implied by Condition \ref{c:recurrence}-$(\alpha, \beta)$, which in turn is implied by the stronger condition
\begin{align*}
\lim_{|y|\to\infty} y \cdot f(y) + \alpha |y|^\beta &= -\infty.
\end{align*}
\end{remark}
One has the infinitesimal generator
\begin{align}
\mathcal{L}&:=f\cdot\nabla_y+\frac12(\tau\tau^T):\nabla_y^2 \label{infinitesimalgenerator}
\end{align}
for the rescaled fast dynamics. Conditions \ref{c:regularity} and \ref{c:recurrencebasic} are enough to guarantee that one has on $\mathcal{Y}$ a unique invariant measure $\mu$ corresponding to the operator $\mathcal{L}$, as discussed for example in \cite{ReyBellet2006}.

\begin{remark}
Therefore, in particular, the process $\ye$, obtained as the solution of an SDE that does not depend on $\xe$, does not explode and is well defined for all times. Meanwhile, Condition \ref{c:regularity} guarantees that the drift coefficient of $\xe$ is Lipschitz continuous in the variable $x$ locally in the variable $y$. Thus one sees that our assumptions are sufficient to guarantee that $\xe$ is well defined on $[0, T]$ (compare the situation with, e.g., \cite[Sections 2 and 4]{pardoux2001poisson}). For general results on existence and uniqueness of solutions of equations with standard and fractional Brownian motions, see for instance \cite{guerranualart, kubiliusfractionalbrownianmotion, kubiliuspsemimartingale, Mishurashevchenko}.
\end{remark}

We now state our main results, the first of which concerns the typical behavior of $\xe$ as
$\varepsilon\to 0$. We prove in Theorem \ref{t:xlimit} that $\xe$ converges in an $L^p$ sense, and at a particular rate, to a deterministic limiting process $\bar{X}$. This implies in particular that one has convergence in probability. The proof is deferred to Section \ref{S:TypicalBehavior}.
\begin{theorem}\label{t:xlimit} Assume Conditions \ref{c:regularity} and \ref{c:recurrence}-$(\alpha,\beta,\gamma)$, where $\alpha \geq 0$, $\beta \geq 2$, $\gamma \geq 0$, and $T \beta \gamma \sup_{y \in \mathcal{Y}} ||\tau(y)||^2 < 2 \alpha$. For any $0 < p < \frac{2 \alpha}{T \beta \gamma \sup_{y \in \mathcal{Y}} ||\tau(y)||^2}$, there is a constant $\tilde K$ such that for $\varepsilon:=(\epsilon,\eta)$ sufficiently small,
\begin{align*}
E\sup_{0\leq t\leq T}\left|\xe_t - \bar X_t\right|^p&\leq \tilde K \left( \sqrt\epsilon^p + \sqrt\eta^p \right),
\end{align*}
where $\bar X$ is the (deterministic) solution of the integral equation
\begin{align*}
\bar X_t &= x_0 + \int^t_0 \bar c(\bar X_s)ds,
\end{align*}
where $\bar c$ is the averaged function
\begin{align*}
\bar c(x) &:= \int_{\mathcal{Y}}c(x,y)d\mu(y).
\end{align*}
\end{theorem}
Our second main result concerns the asymptotic behavior of the
(appropriately-rescaled) fluctuations of $\xe$ around $\bar X$ as $\varepsilon\to 0$. We prove in Theorem \ref{t:fluctuations} that the fluctuations converge in distribution to a particular limit, which we characterize explicitly.
The proof is deferred to Section \ref{S:Fluctuations}. In order to state the theorem, we note that by Theorem 3 in \cite{pardoux2003poisson}, the equations
\begin{equation}
\label{poissonequation0}
\begin{cases}  \displaystyle &\mathcal{L}\Phi(x,y) = -\left( c(x,y)-\bar{c}(x) \right) \\
  \displaystyle &\int_{\mathcal{Y}}\Phi(x,y)d\mu(y)=0
\end{cases}
\end{equation}
admit a unique solution $\Phi$ in the class of functions that grow at most polynomially in $|y|$ as $|y|\to\infty$.
Recall that $\theta^\varepsilon:=\frac{1}{\sqrt\epsilon}(X^\varepsilon-\bar X)$.
\begin{theorem}\label{t:fluctuations} Suppose that one is concerned with a class $\mathcal{C}$ of pairs $\varepsilon:=(\epsilon, \eta)$ for which there exists $\lambda\in[0,\infty)$ such that $\lim_{\varepsilon\to0}\frac{\sqrt\eta}{\sqrt\epsilon}=\lambda$. Assume Conditions \ref{c:regularity} and \ref{c:recurrence}-$(\alpha,\beta,\gamma)$, where $\alpha \geq 0$, $\beta \geq 2$, $\gamma \geq 0$, and $T \beta \gamma \sup_{ y \in \mathcal{Y} } || \tau ( y ) ||^2 < 2 \alpha$. 
With $\Phi$ as in (\ref{poissonequation0}), set $\Sigma_\Phi := (\overline{(\nabla_y\Phi\tau)(\nabla_y\Phi\tau)^T})^{1/2}$. One then has that the family of processes $\{\theta^\varepsilon\}_\varepsilon$ converges in distribution on the space $C([0,T]; \mathcal{X})$ (endowed, as usual, with the topology of uniform convergence) as $\varepsilon\to0$ to the law of the solution $\theta$ of the mixed SDE
\begin{align*}
\begin{cases}
\displaystyle \theta_t = \int^t_0 (\nabla_x \bar c)(\bar X_s) \cdot \theta_s ds + \lambda \int^t_0 \Sigma_\Phi(\bar X_s) d\tilde{B}_s +  \int^t_0 \bar \sigma \delta \tilde{W}^H_s\\
\displaystyle \theta_0=0,
\end{cases}
\end{align*}
where $\bar \sigma := \int_{\mathcal{Y}}\sigma(y)d\mu(y)$, $\tilde{W}^H$ is a fractional Brownian motion with Hurst
index $H$, and $\tilde{B}$ is a standard Brownian
motion independent of $\tilde{W}^H$.
\end{theorem}

\section{First-Order Limit or Typical Behavior}\label{S:TypicalBehavior}
In this section, we focus on proving Theorem \ref{t:xlimit}, which
establishes the first-order limit, or typical behavior, of the slow
component $\xe$ as $\varepsilon\to 0$. In order to make the exposition
easier to follow, we present several supporting lemmata leading up to
an ergodic theorem, Theorem \ref{t:ergodic1}, which is the essential
ingredient in the proof of Theorem \ref{t:xlimit}.

\begin{remark}\label{r:integrals}
While we interpret the stochastic integral $\int_0^t \sigma(\ye_s) dW^H_s$ appearing in
\eqref{Eq:ModelSystem} in the pathwise sense to appeal to existence and uniqueness results in the literature, this pathwise integral coincides in our framework with the divergence integral $\int_0^t \sigma(\ye_s) \delta W^H_s$ (see Appendix \ref{A:Appendix} for details on Malliavin calculus and integration with respect to fBm). Indeed, in view of
\eqref{relationpathwise-divergence}, the two integrals coincide as soon as the integrand is in the kernel of the Malliavin derivative associated with the fractional Brownian motion $W^H$, which is of course true in our case as $\sigma(\ye)$ and $W^H$ are independent. In what follows, we work mainly from the point of view of the divergence integral rather than the pathwise integral, so that we may apply results from the Malliavin stochastic calculus of variations.
\end{remark}
We begin by stating a maximal inequality for the divergence integral $\int^t_0 \sigma(\ye_s) \delta W^H_s$.
\begin{lemma}\label{l:sigmadw}
Assume Conditions \ref{c:regularity} and \ref{c:recurrencebasic}. For any $H^{-1}<p<\infty$, there is a constant $\tilde K$ such that for $\eta$ sufficiently small,
\begin{align*}
E \sup_{0\leq t\leq T} \left| \int^t_0 \sigma(\ye_s) \delta W^H_s \right|^p &\leq \tilde K.
\end{align*}
\end{lemma}
\begin{proof}
Recalling that $\sigma$ is polynomially bounded in its argument and appealing to Lemma 1 in \cite{pardoux2003poisson}, the claim follows from the maximal inequality stated after (2.14) in \cite{nualart2006}.
\end{proof}
We next obtain a preliminary uniform bound on the slow component.
\begin{lemma}\label{l:xbound}
Assume Conditions \ref{c:regularity} and \ref{c:recurrencebasic}. For any $0<p<\infty$, there is a constant $\tilde K$ such that for $\varepsilon := (\epsilon, \eta)$ sufficiently small,
\begin{align*}
E \sup_{0\leq t\leq T} |\xe_t|^p &\leq \tilde K.
\end{align*}
\end{lemma}

\begin{proof} It is enough to prove the lemma for $p\geq2$. Recall that
\begin{align*}
\xe_t & = x_0 + \int^t_0 c(\xe_s, \ye_s) ds + \sqrt\epsilon \int^t_0 \sigma(\ye_s) \delta W^H_s.
\end{align*}

By Condition \ref{c:regularity}, there are constants $K>0$, $q>0$, and $r\in[0,1)$ such that $|c(x,y)|\leq K(1+|x|^r)(1+|y|^q)$. By Lemma 1 in \cite{pardoux2003poisson}, Lemma \ref{l:sigmadw} above, and Young's inequality with conjugate exponents $\frac{1}{r}$ and $\frac{1}{1-r}$, for some constants $C_j$, for $t\in[0,T]$, and for $\varepsilon$ sufficiently small,
\begin{align*}
E \sup_{0\leq s\leq t} |\xe_s|^p &\leq C_1 E \left(|x_0|^p + \int^t_0|c(\xe_s,\ye_s)|^pds + \sup_{0\leq s\leq t}\left|\int^s_0\sigma(\ye_u) \delta W^H_u\right|^p\right)\\
&\leq C_2 \left( 1 + E\int^t_0(1+|\xe_s|^r)^p(1+|\ye_s|^q)^pds \right)\\
&\leq C_3 \left( 1 + E\int^t_0|\xe_s|^{rp}(1+|\ye_s|^q)^pds + E\int^t_0(1+|\ye_s|^q)^pds \right)\\
&\leq C_4 \left( 1 + E\int^t_0|\xe_s|^pds + E\int^t_0(1+|\ye_s|^q)^{\frac{p}{1-r}}ds + E\int^t_0(1+|\ye_s|^q)^pds \right)\\
&\leq C_5 \left( 1 + \int^t_0\sup_{0\leq u\leq s}|\xe_u|^pds \right).
\end{align*}
The proof is complete upon applying the Gr\"onwall inequality.
\end{proof}
Taking together the bound on the slow component in Lemma \ref{l:xbound} and the bound on the fast component in Lemma 1 in \cite{pardoux2003poisson}, we now show that polynomially-bounded measurable functions of $\xe$ and $\ye$ represent classes in $L^p(\Omega\times[0,T])$.
\begin{lemma}\label{l:extendedbound} Assume Conditions \ref{c:regularity} and \ref{c:recurrencebasic}. Let $h$ be a measurable function on $\mathcal{X} \times \mathcal{Y}$ and suppose that constants $K, r, q > 0$ exist for which $|h(x,y)|\leq K(1+|x|^r)(1+|y|^q)$. For any $0<p<\infty$, there is a constant $\tilde K$ such that for $\varepsilon:=(\epsilon,\eta)$ sufficiently small,
\begin{align*}
E\int^T_0|h(\xe_t,\ye_t)|^pdt&\leq\tilde K.\
\end{align*}
\end{lemma}
\begin{proof}
Note that, using our assumption on the function $h$, we can write
\begin{align*}
E\int^T_0|h(\xe_t,\ye_t)|^pdt &\leq E \int^T_0 K^p (1+|\xe_t|^r)^p (1+|\ye_t|^q)^p dt\\
&\leq \frac{K^p}{2}\left( E\int^T_0(1+|\xe_t|^r)^{2p}dt + E\int^T_0(1+|\ye_t|^q)^{2p} dt \right).
\end{align*}
The terms inside the parentheses are bounded respectively by Lemma \ref{l:xbound} above and Lemma 1 in \cite{pardoux2003poisson}, concluding the proof.
\end{proof}
As we have mentioned, the proof of our ergodic theorem relies on having first obtained uniform bounds not only on the slow component $\xe$ but also on its Malliavin derivative $D\xe$ with respect to the fractional Brownian motion $W^H$. The next lemma provides appropriate technical uniform bounds on exponential moments of the fast component, which we will then use to establish the necessary bound on the Malliavin derivative.
\begin{lemma}\label{l:yboundexponential} Assume Conditions \ref{c:regularity} and \ref{c:recurrence}-$(\alpha, \beta)$, where $\alpha > 0$ and $\beta \geq 2$. For any $\nu\geq0$ such that $\nu \beta \sup_{y\in\mathcal{Y}}||\tau(y)||^2 < 2\alpha $, there is a constant $\tilde K$ such that for all $\eta>0$,
\begin{align*}
\sup_{0\leq t\leq T}Ee^{\nu|\ye_t|^\beta}\leq\tilde K.
\end{align*}
\end{lemma}

\begin{proof} Let $\tilde B$ be a standard Brownian motion and let
  $\tilde Y$ denote the solution of the stochastic differential equation
\begin{equation*}
\begin{cases}  \displaystyle &d\tilde Y_t = f(\tilde Y_t) dt + \tau(\tilde Y_t) d\tilde B_t \\
  \displaystyle &\tilde Y_0 = y_0.
\end{cases}
\end{equation*}
Since $\{\ye_t\}_{0\leq t\leq T}$ has the same law as $\{\tilde Y_{t/\eta}\}_{0\leq t\leq T}$, the claim of the lemma is equivalent to the statement that
\begin{align}
\sup_{0\leq t<\infty}Ee^{\nu|\tilde Y_t|^\beta}<\infty\label{e:ueb}.
\end{align}
Fix $N\in\mathbb{N}$ and put $t_N:=\inf\{t\in[0,\infty];|\tilde Y_t|\geq N\}$. Let $\tilde Y_N$ denote the process obtained by halting $\tilde Y$ at time $t_N$, i.e., let $\{\tilde Y_{N,t}\}_{t \geq 0} :=
\{\tilde Y_{t \wedge t_N}\}_{t \geq 0}$. We will show that there is a constant $\tilde K$ such that for all $N\in\mathbb{N}$,
\begin{align}
\sup_{0\leq t<\infty}Ee^{\nu|\tilde Y_{N,t}|^\beta}<\tilde K\label{e:uebN}
\end{align}
and that
\begin{align}
P(\lim_{N\to\infty}t_N=\infty)=1.\label{e:hittingtime}
\end{align}
Taking together (\ref{e:uebN}) and (\ref{e:hittingtime}), (\ref{e:ueb}) follows easily.

To establish (\ref{e:uebN}), choose $\ell>1$ such that $\ell \nu \beta \sup_{y\in\mathcal{Y}}||\tau(y)||^2 \leq 2\alpha $ and apply It\^o's lemma to obtain
\begin{align}
Ee^{\nu|\tilde Y_{N,t}|^\beta} &=
                                 e^{\nu|y_0|^\beta}E\left[e^{\nu\int^t_0\nabla(|\tilde
                                 Y_{N,s}|^\beta)d \tilde Y_{N,s} +
                                 \frac{\nu}{2}\int^t_0\nabla^2(|\tilde
                                 Y_{N,s}|^\beta):(\tau\tau^T)(\tilde
                                 Y_{N,s})ds}\right]\nonumber \\
&= e^{\nu|y_0|^\beta}E\left[e^{I_t} \cdot e^{II_t}\right],\label{e:uebN1}
\end{align}
where, with $1_{d-m}$ denoting the $(d-m)\times(d-m)$ identity matrix,
\begin{align*}
I_t&:=\nu\int^t_0 \Bigg( \beta|\tilde Y_{N,s}|^{\beta-2}\tilde Y_{N,s} \cdot f(\tilde Y_{N,s}) + \frac{\nu \ell}{2(\ell-1)} \beta^2|\tilde Y_{N,s}|^{2\beta-4}|\tilde Y_{N,s} \cdot \tau(\tilde Y_{N,s})|^2 \\
&\hspace{6pc}+ \frac{1}{2}\left(\beta(\beta-2)|\tilde Y_{N,s}|^{\beta-4}\tilde Y_{N,s} \otimes \tilde Y_{N,s} + \beta|\tilde Y_{N,s}|^{\beta-2}1_{d-m}\right):(\tau\tau^T)(\tilde Y_{N,s}) \Bigg) ds,\\
II_t&:=\nu\int^t_0 \beta|\tilde Y_{N,s}|^{\beta-2}\tilde Y_{N,s} \cdot \tau(\tilde Y_{N,s}) d\tilde B_s - \frac{\nu^2 \ell}{2(\ell-1)}\int^t_0 \beta^2|\tilde Y_{N,s}|^{2\beta-4}|\tilde Y_{N,s} \cdot \tau(\tilde Y_{N,s})|^2 ds.
\end{align*}
Applying Young's inequality with conjugate exponents $\ell$ and $\frac{\ell}{\ell-1}$,
\begin{align}
e^{\nu|y_0|^\beta}E\left[e^{I_t} \cdot e^{II_t}\right] &\leq e^{\nu|y_0|^\beta}\left(\frac{1}{\ell}Ee^{\ell I_t} + \frac{\ell-1}{\ell}Ee^{\frac{\ell}{\ell-1}II_t}\right).\label{e:uebN2}
\end{align}

Note that, on the one hand, by Condition \ref{c:recurrence}-$(\alpha, \beta)$,
there is a constant $C$ independent of $N$ such that
$Ee^{\ell I_t} \leq C$, and on the other hand, $e^{\frac{\ell}{\ell-1}II}$ is an exponential martingale with unit mean. Consequently,
\begin{align}
e^{\nu|y_0|^\beta}\left(\frac{1}{\ell}Ee^{\ell I_t} + \frac{\ell-1}{\ell}Ee^{\frac{\ell}{\ell-1}II_t}\right) &\leq e^{\nu|y_0|^\beta}\left(\frac{1}{\ell} C + \frac{\ell-1}{\ell}\right).\label{e:uebN3}
\end{align}

Putting together (\ref{e:uebN1}), (\ref{e:uebN2}), and (\ref{e:uebN3}), we obtain (\ref{e:uebN}) with $\tilde K := \frac{C+\ell-1}{\ell}e^{\nu|y_0|^\beta}$.

It remains to verify (\ref{e:hittingtime}). By \cite[Chapter 6,
Theorem 4.1]{IW}, one may realize on some probability space $(\tilde
\Omega, \tilde{\mathcal{F}}, \tilde P)$ real-valued stochastic
processes $\{\Xi_t\}_{t \geq 0}$ and $\{\Xi^+_t\}_{t \geq 0}$ such
that $\Xi_0=\Xi^+_0=|y_0|^2$,
\begin{equation*}
\tilde P \left( \forall t \geq 0, \sup_{0 \leq s \leq t} \Xi_s \leq \sup_{0 \leq s \leq t} \Xi^+_s \right) = 1,
\end{equation*}
$\Xi$ is equal in law to $|\tilde Y|^2$, and $\Xi^+$ is an It\^o diffusion with generator $a(\xi)\left(b(\xi)\frac{\partial}{\partial \xi}+\frac{1}{2}\frac{\partial^2}{\partial \xi^2}\right)$, where, for $\xi > 0$,
\begin{align*}
a(\xi)&:=\sup_{y\in\mathcal{Y}; |y|^2=\xi} \left(4y^T\tau(y)\tau^T(y)y\right),\\
b(\xi)&:=\sup_{y\in\mathcal{Y}; |y|^2=\xi} \left(\frac{2y\cdot f(y) + |\tau(y)|^2}{4y^T\tau(y)\tau^T(y)y}\right).
\end{align*}

It is enough, then, to show that $\tilde P \left(\lim_{N\to\infty}\inf\{t\in[0,\infty];\Xi^+_t \geq N\}=\infty\right) = 1$. If this were not the case, one would have some fixed $\tilde T>0$ for which the event
\begin{align*}
\tilde A & := \left\{ \sup_{N \in \mathbb{N}} \left( \inf\{t\in[0,\tilde T];\Xi^+_t \geq N\} \right) < \tilde T \right\}
\end{align*}
is such that $\tilde P ( \tilde A ) > 0$. One would then have that for
all $N \in \mathbb{N}$, $E \sup_{0 \leq t \leq \tilde T}|\Xi^+_t| \geq N \tilde P ( \tilde A )$, whence immediately $E \sup_{0 \leq t \leq \tilde T}|\Xi^+_t| = \infty$. It therefore suffices to show that for each fixed $\tilde T > 0$, $E \sup_{0 \leq t \leq \tilde T}|\Xi^+_t| < \infty$.

By Condition \ref{c:recurrence}-$(\alpha, \beta)$, one sees in
particular that $a(\xi)b(\xi)$ is negative for $\xi$ sufficiently
large, say $\xi>\xi_0$. In light of this observation, by \cite[Chapter 6, Theorem 1.1]{IW} coupled with a stopping-and-starting argument it suffices to consider an It\^o diffusion $\tilde \Xi$ with initial value $\tilde \Xi_0 = |y_0|^2\wedge\xi_0$ and generator $\frac{1}{2} a(\xi) \frac{\partial^2}{\partial\xi^2}$, and to show that for each fixed $\tilde T > 0$, $E \sup_{0 \leq t \leq \tilde T}|\tilde\Xi_t| < \infty$.

To this end, suppose that
\begin{align*}
\tilde\Xi_t &= \tilde\Xi_0 + \int^t_0 \sqrt{a(\tilde\Xi_s)}dV_s,
\end{align*}
where $V_s$ is a standard Brownian motion in one dimension. By the Burkholder-Davis-Gundy inequality and the fact that $a(\xi)\leq 4 (\sup_{y\in\mathcal{Y}}||\tau(y)||^2) (1+\xi^2)$, we have, for $K:=16\sup_{y\in\mathcal{Y}}||\tau(y)||^2$,
\begin{align*}
E \sup_{0 \leq s \leq t} \tilde\Xi^2_s & \leq 2\tilde\Xi^2_0 + 2 E \sup_{0 \leq s \leq t} \left| \int^s_0 \sqrt{a(\tilde\Xi_u)}d\tilde{B}_u \right|^2\\
& \leq 2\tilde\Xi^2_0 + 2 K t + 2 K \int^t_0 E \sup_{0 \leq u \leq s} \tilde\Xi^2_u du,
\end{align*}
whence Gr\"onwall's inequality gives
\begin{align*}
E \sup_{0 \leq t \leq \tilde T} \tilde\Xi^2_t & \leq \left(2 \tilde\Xi^2_0 + 2 K \tilde T\right) e^{2 K \tilde T}.
\end{align*}
This completes the proof of the lemma.
\end{proof}
With Lemma \ref{l:yboundexponential} in hand, we are now in a position to establish the necessary bound on the Malliavin derivative $D\xe$ of the slow component $\xe$.

\begin{lemma}\label{l:Dxbound} Assume Conditions \ref{c:regularity} and \ref{c:recurrence}-$(\alpha, \beta, \gamma)$, where $\alpha > 0$, $\beta \geq 2$, $\gamma > 0$, and $T \beta \gamma \sup_{y\in\mathcal{Y}}||\tau(y)||^2 < 2 \alpha$. For any $0 < p < \frac{2 \alpha}{T \beta \gamma \sup_{y\in\mathcal{Y}}||\tau(y)||^2}$, there is a constant $\tilde K$ such that for $\varepsilon:=(\epsilon,\eta)$ sufficiently small,
\begin{align*}
\sup_{0\leq s, t\leq T}E|D_s\xe_t|^p&\leq\tilde K.
\end{align*}
\end{lemma}

\begin{proof} It is enough to prove the lemma for $p>1$. We begin by noting that
\begin{align*}
D_s\xe_t&=\int^t_0\nabla_xc(\xe_u,\ye_u)D_s\xe_udu+\sqrt\epsilon\sigma(\ye_s)\chi_{[0,t]}(s).
\end{align*}
Hence, for $t<s$, $D_s\xe_t=0$, while for $t \geq s$,
\begin{align*}
|D_s\xe_t|^p &= |\sqrt\epsilon\sigma(\ye_s)|^p + \int^t_s p|D_s\xe_u|^{p-2}D_s\xe_u:(\nabla_x c(\xe_u,\ye_u)D_s\xe_u) du\\
&\leq |\sqrt\epsilon\sigma(\ye_s)|^p + \int^t_s p|D_s\xe_u|^p ||\nabla_x c(\xe_u,\ye_u)|| du.
\end{align*}
Applying Gr\"onwall's inequality and Young's inequality with conjugate exponents $\frac{\ell}{\ell-1}$ and $\ell$ yet to be determined, we obtain, for $t \geq s$,
\begin{align*}
|D_s\xe_t|^p &\leq |\sqrt\epsilon\sigma(\ye_s)|^p e^{p \int^t_s ||\nabla_x c(\xe_u,\ye_u)||du}\\
&\leq I + II_t,
\end{align*}
where
\begin{align*}
I &:= \frac{\ell-1}{\ell}|\sqrt\epsilon\sigma(\ye_s)|^{\frac{\ell p}{\ell-1}},\\
II_t &:= \frac{1}{\ell}e^{\ell p \int^t_s ||\nabla_x c(\xe_u,\ye_u)|| du}.
\end{align*}

For any given choice of $\ell>1$, the expected value $E(I)$ of the first summand is easily handled by \cite[Corollary 1]{pardoux2001poisson}, so we proceed to consider the expected value of the second summand. Recalling that, by assumption, $||\nabla_xc(x,y)||\leq \gamma |y|^\beta$ for $|y|$ sufficiently large, we have, for some constant $C>0$, applying Jensen's inequality,
\begin{align*}
EI_t &\leq E\frac1\ell e^{\ell p \int^T_0||\nabla_xc(\xe_u,\ye_u)|| du}\\
&\leq \frac1\ell C^{\ell} + E\frac1\ell e^{\ell p \gamma \int^T_0|\ye_u|^\beta du}\\
&\leq \frac1\ell C^{\ell} + E\frac1{\ell T}\int^T_0 e^{\ell p T \gamma |\ye_u|^\beta}du\\
&\leq \frac1\ell C^{\ell} + \frac1\ell \sup_{0\leq t\leq T} E e^{\ell p T \gamma |\ye_t|^\beta},
\end{align*}
whence the proof is complete upon choosing $\ell>1$ small enough that $\ell p T \beta \gamma \sup_{y\in\mathcal{Y}}||\tau(y)||^2 \leq 2 \alpha$ and appealing to Lemma \ref{l:yboundexponential}.

\end{proof}
The technical ingredients for the ergodic theorem are now in place. Before moving on to the theorem, we present a version of It\^o's lemma adapted to our framework.

\begin{lemma}\label{l:itolemma} For any function $F$ of class $C^2$ on $\mathcal{X} \times \mathcal{Y}$,
\begin{align}
F(\xe_t,\ye_t)-F(x_0,y_0) &= \int^t_0 (\nabla_x F)(\xe_s,\ye_s)
                            \delta\xe_s + \int^t_0 (\nabla_y
                            F)(\xe_s,\ye_s) d\ye_s \label{e:itolemma}
                            \nonumber \\
&\hspace{2pc} + \epsilon \alpha_H \int^t_0 (\nabla^2_x
  F)(\xe_s,\ye_s):\sigma(\ye_s)\left(\int^s_0
  \sigma(\ye_u)(s-u)^{2H-2} du\right) ds \nonumber \\
&\hspace{2pc} + \frac{1}{2 \eta} \int^t_0 (\nabla^2_y F)(\ye_s):(\tau\tau^T)(\ye_s) ds,
\end{align}
where $\alpha_H := H(2H-1)$.
\end{lemma}

\begin{proof} This is a straightforward extension of the well-known
  It\^o formula for the divergence integral (see e.g. \cite[Theorem 8]{alosnualart}).
\end{proof}
\begin{remark}
With $\mathcal{L}$ as in (\ref{infinitesimalgenerator}), equation (\ref{e:itolemma}) may also be written as
\begin{align*}
F(\xe_t,\ye_t)-F(x_0,y_0) &= \frac{1}{\eta}\int^t_0 (\mathcal{L}F)(\xe_s,\ye_s) ds + \int^t_0 (\nabla_x F c)(\xe_s,\ye_s) ds \\
&\hspace{2pc} + \epsilon \alpha_H \int^t_0 (\nabla^2_x F)(\xe_s,\ye_s):\sigma(\ye_s)\left(\int^s_0 \sigma(\ye_u)(s-u)^{2H-2} du\right) ds \\
&\hspace{2pc} + \sqrt{\epsilon} \int^t_0 (\nabla_x F \sigma)(\xe_s,\ye_s) \delta W^H_s + \frac{1}{\sqrt{\eta}} \int^t_0 (\nabla_y F \tau)(\xe_s,\ye_s) dB_s.
\end{align*}
\end{remark}
We are now ready to state and prove our ergodic theorem.

\begin{theorem}\label{t:ergodic1} Assume Conditions \ref{c:regularity} and \ref{c:recurrence}-$(\alpha,\beta,\gamma)$, where $\alpha > 0$, $\beta \geq 2$, $\gamma > 0$, and $T \beta \gamma \sup_{y \in \mathcal{Y}} ||\tau(y)||^2 < 2 \alpha$. Let $h$ be a differentiable function on $\mathcal{X} \times \mathcal{Y}$ and suppose that constants $K, r, q > 0$ exist for which $|h(x,y)|\leq K(1+|x|^r)(1+|y|^q)$. Suppose further that each derivative of $h$ up to second order is locally H\"older continuous in $y$ uniformly in $x$, with absolute value growing at most polynomially in $|y|$ as $|y|\to\infty$. For any $0 < p < \frac{2 \alpha}{T \beta \gamma \sup_{y \in \mathcal{Y}} ||\tau(y)||^2}$, there is a constant $\tilde K$ such that for $\varepsilon:=(\epsilon,\eta)$ sufficiently small,
\begin{align*}
E\sup_{0\leq t\leq T}\left|\int^t_0\left(h(\xe_s,\ye_s)-\bar h(\xe_s)\right)ds\right|^p&\leq \tilde K \sqrt\eta^p,
\end{align*}
where $\bar h(x)$ is the averaged function $\int_{\mathcal{Y}}h(x,y)d\mu(y)$.
\end{theorem}

\begin{proof} It is enough to prove the theorem for $p\geq2$. By \cite[Theorem 3]{pardoux2003poisson}, the equations
\begin{align*}
\begin{cases}
&\mathcal{L}\Phi(x,y)=h(x,y)-\bar{h}(x)\\
&\int_{\mathcal{Y}}\Phi(x,y)d\mu(y)=0
\end{cases}
\end{align*}
admit a unique solution $\Phi$ in the class of functions that grow at most polynomially in $|y|$ as $|y|\to\infty$. Applying Lemma \ref{l:itolemma} with $F=\Phi$ and rearranging terms gives
\begin{align*}
\int^t_0 \Big( h(\xe_s,\ye_s)-\bar{h}(\xe_s) \Big) \hspace{0.2pc} ds &= \sqrt{\eta} \hspace{0.2pc} \Bigg( \sqrt{\eta} \hspace{0.2pc} \Big(\Phi(\xe_t,\ye_t)-\Phi(x_0,y_0)\Big) - \sqrt{\eta} \int^t_0 (\nabla_x\Phi c)(\xe_s,\ye_s) ds\\
&\hspace{2pc} - \epsilon\sqrt{\eta} \hspace{0.2pc} \alpha_H \int^t_0 (\nabla^2_x\Phi\sigma)(\xe_s,\ye_s)\cdot\left(\int^s_0 \sigma(\ye_u)(s-u)^{2H-2}du\right)ds\\
&\hspace{2pc} - \sqrt{\epsilon\eta} \int^t_0 (\nabla_x\Phi\sigma)(\xe_s,\ye_s)\delta W^H_s - \int^t_0 (\nabla_y\Phi\tau)(\xe_s,\ye_s) dB_s \Bigg),
\end{align*}
where $\alpha_H := H(2H-1)$; hence, for $\varepsilon$ sufficiently small,
\begin{align}
&E\sup_{0\leq t\leq T}\Bigg|\int^t_0\Big( h(\xe_s,\ye_s)-\bar{h}(\xe_s) \Big) \hspace{0.2pc} ds\Bigg|^p\nonumber\\
&\hspace{2pc} \leq 5^p\sqrt{\eta}^p \hspace{0.2pc} \Bigg( E \sup_{0\leq t\leq T} \sqrt{\eta}^p \hspace{0.2pc} \Big|\Phi(\xe_t,\ye_t)-\Phi(x_0,y_0)\Big|^p + \sqrt{\eta}^p E \int^T_0 |(\nabla_x\Phi c)(\xe_s,\ye_s)|^p ds\nonumber\\
&\hspace{3pc} + (\epsilon\sqrt{\eta})^p \hspace{0.2pc} \alpha_H^p E
  \int^T_0
  \left|(\nabla^2_x\Phi\sigma)(\xe_s,\ye_s)\cdot\left(\int^s_0
  \sigma(\ye_u)(s-u)^{2H-2}du\right)\right|^p ds\nonumber \label{term:sigmasu}\\
&\hspace{3pc} + \sqrt{\epsilon\eta}^p E \sup_{0\leq t\leq T} \left| \int^t_0 (\nabla_x\Phi\sigma)(\xe_s,\ye_s)\delta W^H_s \right|^p + E \sup_{0\leq t\leq T} \left| \int^T_0 (\nabla_y\Phi\tau)(\xe_s,\ye_s) dB_s \right|^p \Bigg).
\end{align}
It remains to show that the expected value terms inside the parentheses are bounded uniformly in $\varepsilon$ sufficiently small. Recalling the stochastic representation of $\Phi$ in \cite{pardoux2001poisson, pardoux2003poisson} and the argument in the proof of \cite[Theorem 3]{pardoux2003poisson}, the function $\Phi$ itself and all of the derivatives of $\Phi$ that appear are continuous in $x$ and $y$ and bounded by expressions of the form $K(1+|x|^r)(1+|y|^q)$. Thus, the term $E \sup_{0\leq t\leq T} \sqrt{\eta}^p \hspace{0.2pc} \Big|\Phi(\xe_t,\ye_t)-\Phi(x_0,y_0)\Big|^p$ is bounded by Lemma \ref{l:xbound} above and \cite[Corollary 1]{pardoux2001poisson}. Meanwhile, the Riemann integral terms are bounded by Lemma \ref{l:extendedbound} (separating the two factors of the product $\sigma(\ye_u)(s-u)^{2H-2}$ that appears in (\ref{term:sigmasu}) by, for example, Young's inequality), and the ordinary Brownian integral term is bounded by the Burkholder-Davis-Gundy inequality and Lemma \ref{l:extendedbound}.

It remains only to bound the stochastic integral term $E \sup_{0\leq
  t\leq T} \left| \int^t_0 (\nabla_x\Phi\sigma)(\xe_s,\ye_s)\delta
  W^H_s \right|^p$. The maximal inequality stated after (2.14) in
\cite{nualart2006} gives a satisfactory bound. To complete the proof of the theorem, it therefore suffices to verify that the integrand $(\nabla_x\Phi\sigma)(\xe,\ye)$ is in the appropriate class, i.e., that
\begin{align*}
E\left(|(\nabla_x\Phi\sigma)(\xe,\ye)|^p_{L^{1/H}([0,T])}+|D(\nabla_x\Phi\sigma)(\xe,\ye)|^p_{L^{1/H}([0,T]^2)}\right)<\infty,
\end{align*}
uniformly in $\varepsilon$ sufficiently small.

The first summand, $E|(\nabla_x\Phi\sigma)(\xe,\ye)|^p_{L^{1/H}([0,T])}$, is easily handled by Lemma \ref{l:xbound} above and \cite[Corollary 1]{pardoux2001poisson}, so we proceed to consider the second summand. For this, we have by Jensen's inequality and then Young's inequality with conjugate exponents $\ell$ and $\frac{\ell}{\ell-1}$,
\begin{align*}
E|D(\nabla_x\Phi\sigma)(\xe,\ye)|^p_{L^{1/H}([0,T]^2)}&=E\left(\int^T_0\int^T_0\left|(\nabla^2_x\Phi\sigma)(\xe_t,\ye_t) \cdot D_s\xe_t\right|^{\frac1H}dsdt\right)^{pH}\\
&\leq T^{2pH-2}E\int^T_0\int^T_0|D_s\xe_t|^p|\nabla^2_x\Phi(\xe_t,\ye_t)\sigma(\ye_t)|^pdsdt\\
&\leq I + II,
\end{align*}
where
\begin{align*}
I&:=\frac{T^{2pH-2}}\ell E\int^T_0\int^T_0|D_s\xe_t|^{\ell p}dsdt,\\
II&:=\frac{T^{2pH-2}(\ell-1)}\ell TE\int^T_0|\nabla^2_x\Phi(\xe_t,\ye_t)\sigma(\ye_t)|^{\frac\ell{\ell-1}p}dt.
\end{align*}
Choosing $\ell>1$ sufficiently small, the term $I$ is handled by Lemma \ref{l:Dxbound}, while the term $II$ is easily handled by Lemma \ref{l:xbound} above and \cite[Corollary 1]{pardoux2001poisson}. This completes the proof of the theorem.
\end{proof}
We are now ready to prove Theorem \ref{t:xlimit} based on the above
ergodic theorem.
\begin{proof}[Proof of Theorem \ref{t:xlimit}]
 Given Theorem \ref{t:ergodic1}, the argument is as in the proof of \cite[Theorem 1]{gs:statinf}.
\end{proof}

\section{Second-Order Limit}\label{S:Fluctuations}

This section is dedicated to proving Theorem \ref{t:fluctuations},
which establishes a limit in distribution of the (appropriately-rescaled) fluctuations of $\xe$ about its deterministic typical behavior $\bar X$. We denote the fluctuations process by
$\theta^\varepsilon := \frac{1}{\sqrt\epsilon}(\xe-\bar X)$. We then have the following decomposition:
\begin{align*}
\theta^\varepsilon &= I^\varepsilon + II^\varepsilon + III^\varepsilon,
\end{align*}
where, for $0 \leq t \leq T$,
\begin{align*}
&I^\varepsilon_t := \frac{1}{\sqrt\epsilon} \int^t_0 \left( \bar c(\xe_s) - \bar c(\bar X_s) \right) ds,\\
&II^\varepsilon_t := \frac{1}{\sqrt\epsilon} \int^t_0 \left( c(\xe_s, \ye_s) - \bar c(\xe_s) \right) ds,\\
&III^\varepsilon_t := \int^t_0 \sigma(\ye_s) \delta W^H_s.
\end{align*}

\begin{lemma}\label{l:tightness} Suppose that one is concerned with a class $\mathcal{C}$ of pairs $\varepsilon := (\epsilon, \eta)$ for which there exists $\lambda \in [0, \infty)$ for which $\lim_{\varepsilon\to0}\frac{\sqrt\eta}{\sqrt\epsilon}=\lambda$. Assume Conditions \ref{c:regularity} and \ref{c:recurrence}-$(\alpha,\beta,\gamma)$, where $\alpha \geq 0$, $\beta \geq 2$, $\gamma \geq 0$, and $T \beta \gamma \sup_{ y \in \mathcal{Y} } || \tau ( y ) ||^2 < 2 \alpha$. For some $\epsilon_0>0$, one has tightness of the family of distributions on $C([0,T]; \mathcal{X}^{4})$ (endowed, as usual, with the topology of uniform convergence) associated with the family of processes $\left\{\Theta^\varepsilon\right\}_{\varepsilon\in\mathcal{C}; 0<\epsilon<\epsilon_0}$, where $\Theta^\varepsilon:=(\theta^\varepsilon, I^\varepsilon, II^\varepsilon, III^\varepsilon)$.
\end{lemma}

\begin{proof} As in Theorem 7.3 in \cite{billingsley}, a family $\left\{\Pi^\varepsilon\right\}_{\varepsilon\in\mathcal{C}; 0<\epsilon<\epsilon_0}$ represents a tight family of distributions if and only if for all $\zeta>0$,
\begin{align}
\exists N\in\mathbb{N}; &\sup_{\varepsilon\in\mathcal{C}; 0<\epsilon<\epsilon_0} P \left[ \sup_{0\leq t\leq T}|\Pi^\varepsilon_t| \geq N \right] \leq \zeta \label{firstcriterion}
\end{align}
and
\begin{align}
\forall M\in\mathbb{N}, &\lim_{\rho\to0}\sup_{\varepsilon\in\mathcal{C};0<\epsilon<\epsilon_0} P \left[ \sup_{0 \leq t_1 < t_2 \leq T, |t_1 - t_2| < \rho} |\Pi^\varepsilon_{t_1} - \Pi^\varepsilon_{t_2}| \geq \zeta, \sup_{0 \leq t \leq T}|\Pi^\varepsilon_t| \leq M \right].\label{secondcriterion}
\end{align}
Applying the triangle inequality in conjunction with this characterization, it is enough to show that $\epsilon_0>0$ may be chosen so that each family $\left\{I^\varepsilon\right\}_{\varepsilon\in\mathcal{C}; 0<\epsilon<\epsilon_0}$, $\left\{II^\varepsilon\right\}_{\varepsilon\in\mathcal{C}; 0<\epsilon<\epsilon_0}$, $\left\{III^\varepsilon\right\}_{\varepsilon\in\mathcal{C}; 0<\epsilon<\epsilon_0}$ represents a tight family of distributions.

Let us first consider separately the family $\left\{I^\varepsilon\right\}_{\varepsilon}$.
\begin{align}
I^\varepsilon_t &= \int^t_0 (\nabla_x \bar c)(\bar X_s) \cdot \theta^\varepsilon_s \hspace{0.5pc} ds + \int^t_0 \left[ (\nabla_x \bar c)(X^{\varepsilon, \dagger}_s) - (\nabla_x \bar c)(\bar X_s) \right] \cdot \theta^\varepsilon_s \hspace{0.5pc} ds \nonumber\\
&=: \int^t_0 (\nabla_x \bar c)(\bar X_s) \cdot \theta^\varepsilon_s \hspace{0.5pc} ds + \mathcal{R}^\varepsilon_{I,t},\label{firstterm}
\end{align}
where $\mathcal{R}^\varepsilon_{I,t} := \int^t_0 \left[ (\nabla_x \bar c)(X^{\varepsilon, \dagger}_s) - (\nabla_x \bar c)(\bar X_s) \right] \cdot \theta^\varepsilon_s \hspace{0.5pc} ds$ and $X^{\varepsilon, \dagger}_s$ is an appropriately-chosen point on the line segment connecting $\xe_s$ with $\bar X_s$.

By Theorem \ref{t:xlimit}, one may choose an $\epsilon_0>0$ for which
$\sup_{0 \leq t \leq T} \left| \theta^\varepsilon_t \right|$ is
bounded in probability uniformly in $\varepsilon\in\mathcal{C}; 0 < \epsilon < \epsilon_0$. The
criteria \eqref{firstcriterion} and \eqref{secondcriterion} are then obviously satisfied with $\left\{t \mapsto \int^t_0 (\nabla_x \bar c)(\bar X_s) \cdot \theta^\varepsilon_s \hspace{0.5pc} ds\right\}_{\varepsilon\in\mathcal{C}; 0 < \epsilon < \epsilon_0}$ in the role of $\left\{\Pi^\varepsilon\right\}_{\varepsilon\in\mathcal{C}; 0 < \epsilon < \epsilon_0}$, whence it follows that $\left\{t \mapsto \int^t_0 (\nabla_x \bar c)(\bar X_s) \cdot \theta^\varepsilon_s \hspace{0.5pc} ds\right\}_{\varepsilon\in\mathcal{C}; 0 < \epsilon < \epsilon_0}$ is tight. Meanwhile,
\begin{align*}
\mathcal{R}^\varepsilon_{I,t} &:= \int^t_0 \left[ (\nabla_x \bar c)(X^{\varepsilon, \dagger}_s) - (\nabla_x \bar c)(\bar X_s) \right] \cdot \theta^\varepsilon_s \hspace{0.5pc} ds
\end{align*}
vanishes in probability uniformly in $t\in[0,T]$ as $\varepsilon\to
0$. In order to see this, note firstly that
\begin{equation*}
\int^t_0 \left[ (\nabla_x \bar c)(X^{\varepsilon, \dagger}_s) - (\nabla_x \bar c)(\bar X_s) \right] ds
\end{equation*}
vanishes in probability by compactness of $[0,T]$, continuity of $\nabla_x \bar c$, and Theorem \ref{t:xlimit}, and secondly that $\sup_{0 \leq t \leq T} \left| \theta^\varepsilon_t \right|$ is bounded in probability by Theorem \ref{t:xlimit}. It is easy to deduce that $\left\{\mathcal{R}^\varepsilon_{I}\right\}_{\varepsilon\in\mathcal{C}; 0 < \epsilon < \epsilon_0}$ is tight. It follows then that $\left\{I^\varepsilon\right\}_{\varepsilon\in\mathcal{C}; 0 < \epsilon < \epsilon_0}$ is tight.

Let us now consider separately the family $\left\{II^\varepsilon\right\}_{\varepsilon}$. As mentioned before, by Theorem 3 in \cite{pardoux2003poisson}, the equations
\begin{equation}
\begin{cases}
&\mathcal{L}\Phi(x,y) = -\left( c(x,y)-\bar{c}(x) \right) \label{poissonequation} \\
&\int_{\mathcal{Y}}\Phi(x,y) d\mu(y)=0\nonumber
\end{cases}
\end{equation}
admit a unique solution $\Phi$ in the class of functions that grow at most polynomially in $|y|$ as $|y|\to\infty$. Applying Lemma \ref{l:itolemma} with $F=\Phi$ and rearranging terms gives
\begin{align}
II^\varepsilon_t &:= \frac{1}{\sqrt{\epsilon}} \int^t_0 \Big( c(\xe_s,\ye_s)-\bar{c}(\xe_s) \Big) \hspace{0.2pc} ds \nonumber\\
&= \frac{\sqrt\eta}{\sqrt\epsilon} \hspace{0.2pc} \Bigg( \sqrt{\eta} \hspace{0.2pc} \Big(\Phi(x_0,y_0)-\Phi(\xe_t,\ye_t)\Big) + \sqrt{\eta} \int^t_0 (\nabla_x\Phi c)(\xe_s,\ye_s) ds\nonumber\\
&\hspace{2pc} + \epsilon\sqrt{\eta} \hspace{0.2pc} \alpha_H \int^t_0 (\nabla^2_x\Phi\sigma)(\xe_s,\ye_s)\cdot\left(\int^s_0 \sigma(\ye_u)(s-u)^{2H-2}du\right)ds\nonumber\\
&\hspace{2pc} + \sqrt{\epsilon\eta} \int^t_0 (\nabla_x\Phi\sigma)(\xe_s,\ye_s)\delta W^H_s + \int^t_0 (\nabla_y\Phi\tau)(\xe_s,\ye_s) dB_s \Bigg)\nonumber\\
&=: \frac{\sqrt\eta}{\sqrt\epsilon} \int^t_0 (\nabla_y\Phi\tau)(\xe_s,\ye_s) dB_s + \mathcal{R}^\varepsilon_{II, t},\label{secondterm}
\end{align}
where $\alpha_H := H(2H-1)$.

The first summand, $\frac{\sqrt\eta}{\sqrt\epsilon} \int^t_0 (\nabla_y\Phi\tau)(\xe_s,\ye_s) dB_s$, converges in distribution to $\lambda \int^t_0 \Sigma_\Phi(\bar X_s) d\tilde{B}_s$, where $\Sigma_\Phi := (\overline{(\nabla_y\Phi\tau)(\nabla_y\Phi\tau)^T})^{1/2}$ and $\tilde{B}$ is a standard Brownian motion. Meanwhile, by arguments as in the proof of Theorem \ref{t:ergodic1}, $\mathcal{R}^\varepsilon_{II, t}$ vanishes in probability uniformly in $t\in[0, T]$ as $\varepsilon\to0$. It is easy to deduce that $\left\{II^\varepsilon\right\}_{\varepsilon\in\mathcal{C}; 0 < \epsilon < \epsilon_0}$ is tight.

As for the family $\left\{III^\varepsilon\right\}_{\varepsilon}$, we have with  $\bar \sigma := \int_{\mathcal{Y}}\sigma(y)d\mu(y)$,
\begin{align}
III^\varepsilon_t &= \int^t_0 \bar\sigma \delta W^H_s + \int^t_0 \left( \sigma(\ye_s) - \bar\sigma \right) \delta W^H_s\nonumber\\
&:= \int^t_0 \bar\sigma \delta W^H_s + \mathcal{R}^\varepsilon_{III, t}.\label{thirdterm}
\end{align}

We claim that $\mathcal{R}^\varepsilon_{III, t}$ vanishes in probability uniformly in $t\in[0,T]$ as $\varepsilon\to 0$, or what is the same, as $\eta\to 0$. As in the proof of Lemma \ref{l:yboundexponential}, let $\tilde B$ be a standard Brownian motion, in this case assumed to be independent of $W^H$, and let
  $\tilde Y$ denote the solution of the stochastic differential equation
\begin{equation*}
\begin{cases}  \displaystyle &d\tilde Y_t = f(\tilde Y_t) dt + \tau(\tilde Y_t) d\tilde B_t \\
  \displaystyle &\tilde Y_0 = y_0.
\end{cases}
\end{equation*}
Since $\{\ye_t\}_{0\leq t\leq T}$ has the same law as $\{\tilde Y_{t/\eta}\}_{0\leq t\leq T}$ and both are independent of $W^H$, for the purposes of this argument, one may work with either process. Thus, if we had assumed $\sigma$ to be uniformly bounded and $Y^\eta$ to begin at time $t=0$ in stationarity, then Theorem 4.15 in \cite{hl:averagingdynamics} would apply directly to establish the claim. Although we do not make these assumptions, our conditions nevertheless suffice to recover the desired convergence in probability, as we now proceed to explain.

The proof of \cite[Theorem 4.15]{hl:averagingdynamics} relies on \cite[Lemma 4.10]{hl:averagingdynamics}, in which the crucial statements are based on certain decay rates for the associated Markov semigroup. We will verify the same bounds in our framework. By \cite[Proposition 1]{pardoux2001poisson}, the uniform nondegeneracy of $\tau\tau^T$ in Condition \ref{c:regularity} together with the dissipativity of $f$ in Condition \ref{c:recurrencebasic} or \ref{c:recurrence} allow us to conclude not only that $\tilde Y$ has a unique invariant measure with finite moments of all orders, but also that for any initial condition $y_0 \in \mathcal{Y}$, each moment of $\tilde Y_t$ may be bounded uniformly in $t \geq 0$. Thus, polynomial bounds on $\sigma$ are enough to obtain uniform bounds on moments of the diffusion coefficient. Moreover, denoting the invariant measure of $\tilde Y$ by $\mu$ and the distribution of $\tilde Y_t$ by $\mu^{y_0}_t$, Condition \ref{c:recurrence} implies exponential decay as $t\to\infty$ of the total variation distance $\textrm{var}(\mu_t-\mu)$ (see for example \cite{ReyBellet2006}). Taking all of this together, one has then
\begin{align}
\left|E\left(\sigma(\ye_t) - \bar\sigma\right)\right|&=\left|E\left(\sigma(\tilde{Y}_{t/\eta}) - \bar\sigma\right)\right|\nonumber\\
&=\left|\int_{\mathcal{Y}}\sigma(y)d(\mu^{y_0}_{t/\eta}-\mu)(y)\right|\nonumber\\
&\leq \left(\int_{\mathcal{Y}}|\sigma(y)|^{p}d(\mu^{y_0}_{t/\eta}+\mu)(y)\right)^{1/p}\left(\int_{\mathcal{Y}}d|\mu^{y_0}_{t/\eta}-\mu|(y)\right)^{1/r}\nonumber\\
&\leq \left(\int_{\mathcal{Y}}K(1+|y|^{p q})d(\mu^{y_0}_{t/\eta}+\mu)(y)\right)^{1/p}\left(\textrm{var}(\mu^{y_0}_{t/\eta}-\mu)\right)^{1/r}\nonumber\\
&\leq C_1 e^{-C_2 \frac{t}{\eta}},\nonumber
\end{align}
where $C_1$ and $C_2$ are finite positive constants that depend neither on $\eta$ nor on $t$.

Therefore, in light of this exponential decay, the arguments of \cite{hl:averagingdynamics} carry over to our setting, and we conclude as desired that $\mathcal{R}^\varepsilon_{III, t}=\int^t_0 \left( \sigma(\ye_s) - \bar\sigma \right) \delta W^H_s$ vanishes in probability uniformly in $t\in[0,T]$ as $\eta\to0$. Details are omitted due to the similarity of the argument. It is then
easy to deduce that $\left\{III^\varepsilon\right\}_{\varepsilon\in\mathcal{C}; 0 < \epsilon < \epsilon_0}$ is tight.
\end{proof}
We are now ready to prove Theorem \ref{t:fluctuations}.
\begin{proof}[Proof of Theorem \ref{t:fluctuations}]
Recall that one is concerned with a class $\mathcal{C}$ of pairs $\varepsilon$. It suffices to show that any sequence of values of $\varepsilon$ in this class tending to $0$ admits a subsequence along which the $\theta^\varepsilon$ converge in distribution to the law of $\theta$. Let us therefore consider now an arbitrary sequence $\{\varepsilon_n\}^\infty_{n=1}\subset\mathcal{C}$ tending to $0$. By Lemma \ref{l:tightness}, passing to a subsequence $\{\varepsilon_{n_k}\}^\infty_{k=1}$, we may suppose that $\{(\theta^{\varepsilon_{n_k}}, I^{\varepsilon_{n_k}}, II^{\varepsilon_{n_k}}, III^{\varepsilon_{n_k}})\}^\infty_{k=1}$ is convergent in distribution. By the Skorohod representation theorem, there is a probability space $(\tilde{\Omega}, \tilde{\mathcal{F}}, \tilde{P})$ supporting stochastic processes $\{(\tilde{\theta}^{\varepsilon_{n_k}}, \tilde{I}^{\varepsilon_{n_k}}, \tilde{II}^{\varepsilon_{n_k}}, \tilde{III}^{\varepsilon_{n_k}})\}^\infty_{k=1}$ equal in distribution to $\{(\theta^{\varepsilon_{n_k}}, I^{\varepsilon_{n_k}}, II^{\varepsilon_{n_k}}, III^{\varepsilon_{n_k}})\}^\infty_{k=1}$, as well as a limiting stochastic process $(\tilde{\theta}, \tilde{I}, \tilde{II}, \tilde{III})$ to which the former converge almost surely as $k$ tends to infinity.

In light of the decompositions \eqref{firstterm}, \eqref{secondterm}, \eqref{thirdterm} and the limits identified in the proof of Lemma \ref{l:tightness} together with uniqueness of the equation defining $\tilde{\theta}$,  we conclude that $\tilde{\theta}$ must be equal in distribution to $\theta$, which completes the proof of the theorem.

\end{proof}

\section{An Extension of the Model}\label{S:Extensions}

We now consider an extension of the model. Consider
\begin{equation}
\label{e:extendedmodel}
\begin{cases} d\xe_t =
\frac{\sqrt{\epsilon}}{\sqrt{\eta}}b(\xe_t, \yeps_t)dt + c(\xe_t, \yeps_t) dt + \sqrt\epsilon \sigma(\yeps_t) \delta W^H_t \\ d\yeps_t = \frac{1}{\eta} f(\yeps_t) dt + \frac{1}{\sqrt{\epsilon\eta}} g(\yeps_t) dt
+ \frac{1}{\sqrt\eta} \tau(\yeps_t) dB_t \\
\xe_0=x_0\in\mathcal{X}, \hspace{1pc} \yeps_0=y_0\in\mathcal{Y}.
\end{cases}
\end{equation}
As in Theorem \ref{t:fluctuations}, we suppose that one is concerned with a class $\mathcal{C}$ of pairs $\varepsilon:=(\epsilon, \eta)$ for which there exists $\lambda\in[0, \infty)$ such that $\lim_{\varepsilon\to0}\frac{\sqrt\eta}{\sqrt\epsilon}=\lambda$. Notice that if $\lambda=0$ then the term $\frac{\sqrt{\epsilon}}{\sqrt{\eta}}b(\xe, \yeps)$ is asymptotically singular. Accordingly, we distinguish two possibilities:
\begin{enumerate}[(i)]
	\item $\lambda = 0$, the `first regime' or `homogenization regime'
	\item $\lambda \in (0, \infty)$, the `second regime' or `averaging regime.'
\end{enumerate}

The extended model \eqref{e:extendedmodel} is particularly relevant when, for example, a fast intermediate scale forms part of the slow component. In the literature this is sometimes referred to as the homogenization regime (see for example \cite{pavliotis2008multiscale} or \cite{spiliopoulos2014fluctuation} for related examples in the framework of perturbation by standard Brownian motion rather than fractional Brownian motion). The scaling in front of the term corresponding to the coefficient $g$ is that which results in a nontrivial limiting contribution in the event that additional intermediate fast scales form part of the main fast component.

We introduce in Condition \ref{c:regularityadditional} our growth and regularity conditions for the new coefficients in the extended model.

\begin{condition}\label{c:regularityadditional}\hspace{1pc}
\hspace{1pc}\newline
\begin{enumerate}[-]
	\item $b$ satisfies the same smoothness and growth conditions as $c$.
	\item In the first regime, $b(x,y)=b(y)$ is a function of the fast variable only and not the slow, and $b$ and its derivatives grow at most polynomially.
	\item In the first regime, $g$ satisfies the same conditions as $c$ does in terms of the $y-$dependence; in the second regime, $g$ satisfies the same conditions as $f$.
\end{enumerate}
\end{condition}

We have as before a basic condition of recurrence type on the fast component, yielding ergodic behavior.

\begin{condition}\label{c:recurrenceextendedbasic}
\begin{align}
\lim_{|y|\to\infty} y \cdot ( f + \lambda g ) (y) &= -\infty;
\end{align}
\end{condition}

As before we shall in fact assume a stronger recurrence condition for our main results.

\begin{condition}\label{c:recurrenceextended}
For real constants $\alpha>0$, $\beta\geq2$, and $\gamma>0$ we shall write:
\begin{enumerate}[-]
	\item Condition \ref{c:recurrenceextended}-$(\alpha,\beta)$: there is a neighborhood $\Lambda$ of $\lambda$ in $[0,\infty)$ such that one has
	\begin{align*}
	\sup_{\tilde{\lambda}\in\Lambda} y \cdot ( f + \tilde{\lambda} g )(y) + \alpha |y|^\beta + \frac12 (\beta-2+d-m)\sup_{\tilde y\in\mathcal{Y}}|\tau(\tilde y)|^2 \leq 0
	\end{align*}
	for $|y|$ sufficiently large.
	\item Condition \ref{c:recurrenceextended}-$(\alpha,\beta,\gamma)$:
	Condition \ref{c:recurrenceextended}-$(\alpha,\beta)$ holds and, moreover, one has, in the first regime, $||\nabla_x c(x,y)|| \leq \gamma|y|^\beta$ for $|y|$ sufficiently large, and in the second regime, perhaps for a smaller neighborhood $\Lambda$, \[
\sup_{\tilde{\lambda}\in\Lambda}||{\tilde{\lambda}}^{-1} \nabla_x b(x,y) + \nabla_x c(x,y)|| \leq \gamma|y|^\beta\]
 for $|y|$ sufficiently large.
\end{enumerate}

\end{condition}
One has the limiting infinitesimal generator
\begin{align}
\mathcal{L}&:=( f + \lambda g ) \cdot\nabla_y+\frac12(\tau\tau^T):\nabla_y^2 \label{limitinginfinitesimalgenerator}
\end{align}
for the rescaled fast dynamics. Conditions \ref{c:regularity}, \ref{c:regularityadditional}, and \ref{c:recurrenceextendedbasic} are enough to guarantee that one has on $\mathcal{Y}$ a unique invariant measure $\mu$ corresponding to the operator $\mathcal{L}$ in equation (\ref{limitinginfinitesimalgenerator}), as discussed for example in \cite{pardoux2001poisson} and \cite{ReyBellet2006}.

In the first regime, a standard centering condition tempers the
asymptotic singularity of the term
$\frac{\sqrt{\epsilon}}{\sqrt{\eta}}b(\xe, \yeps) =
\frac{\sqrt{\epsilon}}{\sqrt{\eta}}b(\yeps)$ (recall that in this
regime, we assume that $b$ is a function of the fast variable only and
not of the slow variable).
\begin{condition}\label{c:centering}
\begin{align*}
\int_{\mathcal{Y}}b(y)d\mu(y)=0.
\end{align*}
\end{condition}

The above conditions are sufficient to derive a first-order limit for the slow process $\xe$ in the context of the extended model (\ref{e:extendedmodel}). In order to obtain a second-order limit, we assume that the convergence of $\frac{\sqrt{\eta}}{\sqrt{\epsilon}}$ to $\lambda$ takes place at a particular rate. Precisely, we assume that $\lim_{\epsilon\to0}\frac{1}{\sqrt{\epsilon}} \left( \frac{\sqrt{\eta}}{\sqrt{\epsilon}} - \lambda \right) =: \kappa\in\mathbb{R}$.

We now sketch how to extend the results of the paper to the extended
model. In the first regime, we must carefully consider the limiting
contribution of the asymptotically-singular term
$\frac{\sqrt{\epsilon}}{\sqrt{\eta}}b(\xe, \yeps) =
\frac{\sqrt{\epsilon}}{\sqrt{\eta}}b(\yeps)$ to the dynamics of the
slow process $\xe$ (recall that in this regime, we assume that $b$ is
a function of the fast variable only and not of the slow one). It turns out that under Condition \ref{c:centering}, the limiting contribution may be captured in terms of the solution of a certain Poisson equation. By Theorem 3 in \cite{pardoux2003poisson}, the equations
\begin{equation}
\begin{cases}
&\mathcal{L}\Psi(y)=-b(y) \label{correction} \\
&\int_{\mathcal{Y}}\Psi(y)d\mu(y)=0\nonumber
\end{cases}
\end{equation}
admit a unique solution $\Psi$ in the class of functions that grow at most polynomially in $|y|$ as $|y|\to\infty$.

In the first regime, we will need the auxiliary drift coefficient
\begin{align}
\varphi_1(x, y) &:= (\nabla_y \Psi \cdot g)(y) + c(x, y) \label{psifirstregime},
\end{align}
where $\Psi$ is as in (\ref{correction}). To play the same role in the second regime, we will need the auxiliary drift coefficient
\begin{align}
\varphi_2(x, y) &:= (\lambda^{-1} b + c)(x, y) \label{psisecondregime}.
\end{align}

Finally, note that in the It\^o formula (Lemma \ref{l:itolemma}), when one considers the extended model, two additional terms, $\frac{\sqrt{\epsilon}}{\sqrt{\eta}}\int^t_0 (\nabla_x F b ) (\xe_s, \yeps_s)ds$ and $\frac{1}{\sqrt{\epsilon\eta}} \int^t_0 (\nabla_y F g ) (\xe_s, \yeps_s) ds$, appear on the right hand side.

We are now ready to state our asymptotic theorems for the extended model.

\begin{theorem}\label{t:xlimitextended} Suppose that one is concerned with a class $\mathcal{C}$ of pairs $\varepsilon:=(\epsilon, \eta)$ for which there exists $\lambda\in[0,\infty)$ such that $\lim_{\varepsilon\to0}\frac{\sqrt\eta}{\sqrt\epsilon}=\lambda$. Let $*\in\{1, 2\}$ indicate respectively the first or second regime. Assume Conditions \ref{c:regularity}, \ref{c:regularityadditional}, and \ref{c:recurrenceextended}-$(\alpha,\beta,\gamma)$, where $\alpha \geq 0$, $\beta \geq 2$, $\gamma \geq 0$, and $T \beta \gamma \sup_{y \in \mathcal{Y}} ||\tau(y)||^2 < 2 \alpha$; in the first regime, assume also Condition \ref{c:centering}. For any $0 < p < \frac{2 \alpha}{T \beta \gamma \sup_{y \in \mathcal{Y}} ||\tau(y)||^2}$, there is a constant $\tilde K$ such that for $\varepsilon:=(\epsilon,\eta)\in\mathcal{C}$ sufficiently small,
\begin{align*}
E\sup_{0\leq t\leq T}\left|\xe_t - \bar X_{*,t}\right|^p&\leq \tilde K \left( \sqrt\epsilon^p + \sqrt\eta^p \right),
\end{align*}
where $\bar X_*$ is the (deterministic) solution of the integral equation
\begin{align*}
\bar X_{*,t} &= x_0 + \int^t_0 \bar \varphi_*(\bar X_{*,s})ds,
\end{align*}
where $\bar \varphi_*$ is obtained, depending on the regime, by averaging (\ref{psifirstregime}) or (\ref{psisecondregime}) with respect to the invariant measure $\mu$.
\end{theorem}

\begin{proof} The proof is almost exactly the same as that of Theorem \ref{t:xlimit}, except that in establishing the analogue of Lemma \ref{l:xbound} in the first regime we must now consider carefully the asymptotically-singular term. Letting $\Psi$ be as in (\ref{correction}), applying the It\^o lemma with $F=\Psi$, and rearranging terms, we obtain
\begin{align}
\frac{\sqrt{\epsilon}}{\sqrt{\eta}}\int^t_0 b(\yeps_s) ds &= \sqrt{\epsilon\eta} \hspace{0.2pc} \Big(\Psi(y_0)-\Psi(\yeps_t)\Big) + \int^t_0 (\nabla_y \Psi g ) (\ye_s) ds + \sqrt{\epsilon} \int^t_0 (\nabla_y\Psi\tau)(\ye_s) dB_s \label{representationofsingularterm} \\
&=: \int^t_0 (\nabla_y \Psi g ) (\yeps_s) ds + \mathcal{R}^\varepsilon_{t}.\nonumber
\end{align}
Here, $\mathcal{R}^\varepsilon$ vanishes. Thus, the proof may proceed as before with $\varphi_1$ in place of $c$.

\end{proof}

To study the distribution of the fluctuations, we must quantify more
precisely the difference between the true drift and the approximate
drift, as was done in formulating the Poisson equation (\ref{poissonequation}) in the proof of Theorem \ref{t:xlimit}. By Theorem 3 in \cite{pardoux2003poisson}, with $*\in\{1,2\}$ indicating the regime, the equations
\begin{align}
\mathcal{L}\Phi_{*}(x,y)&=-(\varphi_{*}(x,y)-\bar{\varphi_{*}}(x)) \label{poissonequationextended} \\
\int_{\mathcal{Y}}\Phi_{*}(x,y)&d\mu(y)=0\nonumber
\end{align}
admit a unique solution $\Phi_{*}$ in the class of functions that grow at most polynomially in $|y|$ as $|y|\to\infty$.

\begin{theorem}\label{t:fluctuationsextended} Suppose that one is concerned with a class $\mathcal{C}$ of pairs $\varepsilon:=(\epsilon,\eta)$ for which there exists $\lambda\in[0,\infty)$ such that $\lim_{\varepsilon\to0}\frac{\sqrt\eta}{\sqrt\epsilon}=\lambda$. Suppose moreover that $\mathcal{C}$ is such that there is a $\kappa\in\mathbb{R}$ for which $\lim_{\varepsilon\to0}\frac{1}{\sqrt{\epsilon}} \left( \frac{\sqrt{\eta}}{\sqrt{\epsilon}} - \lambda \right)=\kappa$. Assume Conditions \ref{c:regularity}, \ref{c:regularityadditional}, and \ref{c:recurrenceextended}-$(\alpha,\beta,\gamma)$, where $\alpha \geq 0$, $\beta \geq 2$, $\gamma \geq 0$, and $T \beta \gamma \sup_{ y \in \mathcal{Y} } || \tau ( y ) ||^2 < 2 \alpha$; in the first regime, assume also Condition \ref{c:centering}. With $\Psi$ and $\Phi_{*}$ respectively as in (\ref{correction}) and (\ref{poissonequationextended}), set $\Sigma_\Psi := (\overline{(\nabla_y\Psi\tau)(\nabla_y\Psi\tau)^T})^{1/2}$ and $\Sigma_{\Phi_{*}} := (\overline{(\nabla_y\Phi_{*}\tau)(\nabla_y\Phi_{*}\tau)^T})^{1/2}$.

In the first regime, the family of processes $\{\theta^\varepsilon\}_\varepsilon$ converges in distribution as $\varepsilon\to0$ to the law of the solution $\theta_1$ of the mixed SDE
\begin{align}
\theta_{1,t} &= \int^t_0 ( \nabla_x \bar {\varphi_1} )(\bar X_{1,s}) \cdot \theta_{1,s} ds + \kappa \int^t_0 \overline{\nabla_y\Phi_{1} \cdot g}(\bar X_{1,s}) ds \nonumber \\
&\hspace{5pc} + \int^t_0 \Sigma_\Psi(\bar X_{1,s}) d\tilde{B}_s + \int^t_0 \bar \sigma \delta \tilde{W}^H_s,\nonumber
\end{align}
where $\tilde{W}^H$ is a fractional Brownian motion with Hurst index $H$ and $\tilde{B}$ is a standard Brownian motion independent of $W^H$. We point out that under our assumptions in this regime we have in fact $\nabla_x \bar {\varphi_1} = \nabla_x \bar c$.

In the second regime, the family of processes $\{\theta^\varepsilon\}_\varepsilon$ converges in distribution as $\varepsilon\to0$ to the law of the solution $\theta_2$ of the mixed SDE
\begin{align}
\theta_{2,t} &= \int^t_0 ( \nabla_x \bar {\varphi_2} )(\bar X_{2,s}) \cdot \theta_{2,s} ds + \kappa \int^t_0 \overline{\nabla_y\Phi_{2} \cdot g}(\bar X_{2,s}) ds - \frac{\kappa}{\lambda^2} \int^t_0 \bar b (\bar X_{2,s}) ds \nonumber\\
&\hspace{5pc} + \lambda \int^t_0 \Sigma_{\Phi_{2}}(\bar X_{2,s}) d\tilde{B}_s + \int^t_0 \bar \sigma \delta \tilde{W}^H_s,\nonumber
\end{align}
where $\tilde{W}^H$ is a fractional Brownian motion with Hurst index $H$ and $\tilde{B}$ is a standard Brownian motion independent of $W^H$.
\end{theorem}

\begin{proof} Given Theorem \ref{t:xlimitextended} and, in particular, the representation (\ref{representationofsingularterm}), the proof is nearly identical to that of Theorem \ref{t:fluctuations}.

\end{proof}

\appendix

\section{Preliminaries}\label{A:Appendix}

\subsection{Fractional Brownian motion}\label{SS:fBm_preliminaries}

A fractional Brownian motion (fBm) is a centered Gaussian process $W^H =
\{ W^H_t \}_{ t \geq 0 } \subset L^2(\Omega)$, characterized by its
covariance function
\begin{equation*}
R_H(t,s) := E (W^H_t W^H_s) = \frac{1}{2} \left( s^{2H} + t^{2H} -
  \left\vert t-s \right\rvert^{2H} \right).
\end{equation*}
It is straightforward to verify that increments of fBm are stationary. The parameter $H \in (0,1)$ is usually referred to as the Hurst exponent, Hurst parameter, or Hurst index.

By Kolmogorov's continuity criterion, such a process admits a modification with continuous sample paths, and we always choose to work with such. In this case one may show in fact that almost every sample path is locally H\"older continuous of any order strictly less than $H$. It is this sense in which it is often said that the value of $H$ determines the regularity of the sample paths.

Note that when $H = \frac{1}{2}$, the covariance function is $R_{\frac{1}{2}}(t,s) = t \wedge s$. Thus, one sees that $W^{\frac{1}{2}}$ is a standard Brownian motion, and in particular that its disjoint increments are independent. In contrast to this, when $H \neq \frac{1}{2}$, nontrivial increments are not independent. In particular, when $H > \frac{1}{2}$, the process exhibits long-range dependence.

Note moreover that when $H \neq \frac{1}{2}$, the fractional Brownian motion is not a semimartingale, and the usual It\^o calculus therefore does not apply.

Another noteworthy property of fractional Brownian motion is that it
is self-similar in the sense that, for any constant $a >0$, the
processes $\left\{ W^H_t\right\}_{ t \geq 0 }$ and $\left\{ a^{-H}
  W^H_{at}\right\}_{ t \geq 0 }$ have the same distribution.

For more details about fractional Brownian motion, we refer the reader
to the monographs \cite{biagini_stochastic_2008, nourdin_selected_2012, nualart_malliavin_2006}.

The self-similarity and long-memory properties of the fractional
Brownian motion make it an interesting and suitable input noise in
many models in various fields such as analysis of financial time series,
hydrology, and telecommunications. However, in order to develop
interesting models based on fractional Brownian motion, one needs a
stochastic calculus with respect to the fBm, which will make use of
the stochastic calculus of variations, or Malliavin calculus, introduced
in the next subsection.

\subsection{Elements of Malliavin calculus}\label{SS:MalliavinCalculus_prelim}
\label{malliavinelements}
We outline here the main tools of Malliavin calculus needed in this
paper. For a complete treatment of this topic, we refer the reader to
\cite{nualart_malliavin_2006}.

Let $W^H = \left\{ W^H_t \right\}_{ t \geq 0 } \subset L^2(\Omega)$ be a fractional Brownian motion with Hurst index $H \in (\frac{1}{2},1)$ and let us fix a time interval $[0,T]$, where $T\in\mathbb{R}_+$.

The formula
\begin{equation*}
\left\langle \chi_{[0,s]},\chi_{[0,t]}\right\rangle_{\mathfrak{H}} := R_H(s,t)
\end{equation*}
induces an inner product on the set $\mathcal{E}$ of step functions on $[0, T]$. We denote by $\mathfrak{H}$ the Hilbert space obtained as the completion of the resulting inner product space.

It can be shown that the formula
\begin{equation}
\label{innerprodfbm}
\left\langle \varphi, \psi \right\rangle_{\mathfrak{H}} := \alpha_H \int_0^T \int_0^T \varphi(r)\psi(u)\left\vert r-u \right\rvert^{2H-2}du dr,
\end{equation}
with $\alpha_H := H(2H-1)$, extends the above inner product from $\mathcal{E}$ to the superset $L^2([0, T])$, and that it is equivalent to define $\mathfrak{H}$ as the completion of this extended inner product space (see e.g. \cite{decreusefond_stochastic_1999}).

Now, the map $\chi_{[0, t]} \mapsto W^H_t$ extends to a linear isometry of Hilbert spaces $\mathfrak{H} \to L^2(\Omega)$. We will denote this map also by $W^H$.

Recall that we are in the setting in which $H > \frac{1}{2}$. While one may interpret $\mathfrak{H}$ as a space of distributions, it has been shown in \cite{pipiras_integration_2000,pipiras_are_2001} that when $H > \frac{1}{2}$, the elements may not be ordinary functions but distributions of negative order. Adapting the inner product \eqref{innerprodfbm}, one can introduce the space $\vert \mathfrak{H} \vert$ of equivalence classes of measurable functions $\varphi$ on $[0, T]$ for which

\begin{equation*}
\left\lVert \varphi \right\rVert_{\vert \mathfrak{H} \vert}^2 := \alpha_H
\int_0^T \int_0^T \left\vert \varphi(r)
\right\rvert \left\vert \varphi(u)
\right\rvert \left\vert r-u
\right\rvert^{2H-2}du dr < \infty,
\end{equation*}
which is in fact a Banach space equipped with this square norm.

It can be shown that one has the following chain of continuous inclusions:
\begin{equation*}
L^2([0, T]) \subset L^{\frac{1}{H}}([0,T]) \subset \vert \mathfrak{H} \vert \subset \mathfrak{H}.
\end{equation*}

Let us now denote by $\mathcal{S}$ the set of smooth cylindrical random variables
of the form $F = f
\left( W^H(\varphi_1), \cdots , W^H(\varphi_n) \right)$, where $n \geq 1$, $\{\varphi_i\}^n_{i=1} \subset
\mathfrak{H}$, and $f \in C_b^{\infty} \left(
  \mathbb{R}^n \right)$ ($f$ and all of its partial derivatives of all orders are bounded functions).

The Malliavin derivative of such a smooth cylindrical random variable
$F$ is defined as the $\mathfrak{H}$-valued random variable given by
\begin{equation*}
DF := \sum_{i=1}^n \frac{\partial f}{\partial x_i} \left( W^H(\varphi_1),
  \cdots, W^H(\varphi_n) \right)\varphi_i.
\end{equation*}
The derivative operator $D$ is a closable operator from $L^2(\Omega)$
into $L^2(\Omega ; \mathfrak{H})$, and we continue to denote by $D$
the closure of the derivative operator, the domain of which we denote by $\mathbb{D}^{1,2}$, and which is a Hilbert space in the Sobolev-type norm
\begin{equation*}
\left\lVert F \right\rVert_{1,2}^2 := E (F^2) + E \left( \left\lVert DF \right\rVert_{\mathfrak{H}}^2 \right).
\end{equation*}

Similarly one obtains a derivative operator $D:\mathbb{D}^{1, 2}(\mathfrak{H}) \to L^2(\Omega; \mathfrak{H} \otimes \mathfrak{H})$ as the closure of $D:L^2(\Omega; \mathfrak{H}) \to L^2(\Omega; \mathfrak{H} \otimes \mathfrak{H})$, and so on.

Note that more generally with $p > 1$ one can analogously obtain $\mathbb{D}^{1, p}$ as Banach spaces of Sobolev type by working with $L^p(\Omega)$.
\\~\\
We can now introduce the divergence operator
$\delta$ as the adjoint of the derivative operator $D$. By definition, an
$\mathfrak{H}$-valued random variable $u \in L^2(\Omega ;
\mathfrak{H})$ is in the domain of $\delta$, which we denote by
$\operatorname{dom}\delta$, if there is a constant $c_u$ for which, for all $F \in \mathbb{D}^{1, 2}$,
\begin{equation*}
\left\vert E \left( \left\langle DF ,u \right\rangle_{\mathfrak{H}} \right)
\right\rvert \leq c_u \left\vert F \right\rvert_{L^2(\Omega)}.
\end{equation*}
For such an element $u$, $\delta(u)$ is defined
by duality as the unique element of $L^2(\Omega)$ such that, for each $F \in \mathbb{D}^{1, 2}$,
\begin{equation*}
E \left( F \delta(u) \right) = E \left( \left\langle DF, u \right\rangle_{\mathfrak{H}} \right).
\end{equation*}
It can be shown that
$\mathbb{D}^{1,2}(\mathfrak{H}) \subset \operatorname{dom}\delta$, and
that for any $u \in \mathbb{D}^{1,2}(\mathfrak{H})$,
\begin{equation*}
E \left( \delta(u)^2 \right) = E \left(\left\lVert u
  \right\rVert_{\mathfrak{H}}^2  \right) + E \left( \left\langle Du,
    (Du)^{*} \right\rangle_{\mathfrak{H}\otimes \mathfrak{H}} \right),
\end{equation*}
where $(Du)^{*}$ is the adjoint of $Du$ in the Hilbert space $\mathfrak{H}\otimes \mathfrak{H}$.\\

\subsection{Multiple Wiener integrals of deterministic functions with respect to fractional Brownian motion}

\subsection{Stochastic integration with respect to fractional Brownian
  motion}\label{Eq:StochasticIntergation_prelim}
~\\
In this subsection we state useful properties of multiple Wiener integrals of elements of $\mathfrak{H}$ with respect the fractional Brownian motion and introduce two main methods used to define stochastic integrals with respect to the fractional Brownian motion. These and other available approaches are collected and discussed in detail in the monograph \cite{biagini_stochastic_2008}.

The first method, introduced in \cite{decreusefond_stochastic_1999}, is based on the stochastic calculus of variations, or Malliavin calculus. Owing to the central role played by the divergence operator introduced in Subsection \ref{malliavinelements}, stochastic integrals of this type are commonly referred to as divergence integrals.

The second approach uses the fact that the H\"older regularity of the paths of fBm with $H > \frac{1}{2}$ is sufficient to allow integration in the sense of Z\"ahle \cite{zahle_integration_1998} or \cite{russo_forward_1993} (see also the classic paper \cite{young_inequality_1936}). Stochastic integrals of this type are often called pathwise integrals.
\begin{remark}
The divergence integral can be formulated for fractional Brownian motion with any $H \in (0, 1)$ whereas the pathwise integral exists only for $H > \frac{1}{2}$. One reason that we restrict attention to the case $H > \frac{1}{2}$ in this work is so that we may make use of known results for both.
\end{remark}

\subsubsection{Divergence integration}\label{SS:DivergenceIntergation}

The definition of the divergence operator as the adjoint of the Malliavin derivative operator suggests interpretation as an integral. Indeed, in the standard Brownian motion case ($H = \frac{1}{2}$), the divergence of an adapted, It\^o-integrable process coincides with its familiar It\^o integral. In general one defines, for $u \in \operatorname{dom} \delta$ and $0 \leq t \leq T$,
\begin{equation*}
\int_0^t u_s \delta W^H_s := \delta(u
\chi_{[0,t]}),
\end{equation*}
which we call the divergence integral of $u$. Note that the divergence integral is always centered in the sense that its expected value is zero.
\\~\\
We shall make use of a maximal inequality for the divergence integral, which we now state. The interested reader is referred to \cite{alosnualart} for more details.

Denote by $\mathbb{L}_H^{1,p}$ the set of elements
$u \in \mathbb{D}^{1,p}(\mathfrak{H})$ for which
\begin{equation*}
E \left(
  \left\vert u \right\rvert_{L^{\frac{1}{H}}([0,T])}^p + \left\vert
    Du \right\rvert_{L^{\frac{1}{H}}([0,T]^2)}^p \right) < \infty.
\end{equation*}
There is a constant $C$ depending only on $H$ and $T$ such that for any $p$ with $pH > 1$ and any $u \in \mathbb{L}_H^{1, p}$,
\begin{equation*}
E \left( \sup_{0 \leq t \leq T} \left\vert \int_0^t u_s \delta W^H_s
\right\rvert^p \right) \leq C \left[ \int_0^T \left\vert E \left( u_s \right)\right\rvert^p ds + \int_0^T E
\left( \int_0^T \left\vert D_s u_r
\right\rvert^{\frac{1}{H}}ds \right)^{pH}dr \right].
\end{equation*}
Here, $D u_r$ is being interpreted as a stochastic process and the subscript $s$ in the notation $D_s u_r$ refers to its parameter. Note that if we denote by $\lambda$ the Lebesgue measure on $[0, T]$, by $P$ the probability measure on $\Omega$, and by $\omega \in \Omega$ the random state, then $D_s u_r$ is defined for $\lambda \times P$-almost-every pair $(s, \omega)$.

\subsubsection{Pathwise integration}\label{SS:PathwiseIntergation}

We present a version of pathwise integration that appears by the name of symmetric stochastic integration in \cite{russo_forward_1993}.

Let $u = \left\{ u_t \right\}_{0 \leq t \leq T}$ be a stochastic
process in $\mathbb{D}^{1,2}(\mathfrak{H})$. If one has that
\begin{equation*}
E \left(\left\lVert u \right\rVert_{\vert \mathfrak{H} \vert}^2 +
  \left\lVert Du \right\rVert_{\vert \mathfrak{H} \vert \otimes \vert
    \mathfrak{H} \vert}^2  \right) < \infty
\end{equation*}
and
\begin{equation*}
\int_0^T \int_0^T \left\vert D_su_t \right\rvert \left\vert t-s
\right\rvert^{2H-2} ds dt < \infty \operatorname{~  a.s.},
\end{equation*}
then the symmetric integral
\begin{equation*}
\int_0^T u_t dW^H_t
\end{equation*}
defined as the limit in probability as $\varepsilon$ tends to zero of
\begin{equation*}
\frac{1}{2\varepsilon} \int_0^T u_s \left( W^H_{(s+\varepsilon) \wedge
  T} - W^H_{(s-\varepsilon) \vee 0}\right)ds
\end{equation*}
exists and for each $t \in [0,T]$,
\begin{equation}
  \label{relationpathwise-divergence}
\int_0^t u_s dW^H_s =  \int_0^t u_s \delta W^H_s + \alpha_H \int_0^t \int_0^T D_r u_s \left\vert s-r
\right\rvert^{2H-2} dr ds.
\end{equation}
Thus one sees how the pathwise and divergence integrals are related to one another. Note in particular that whereas the divergence integral is centered, the pathwise integral generally speaking is not. In the setting of the model in this paper, however, the two integrals coincide.
\begin{remark}
Note that whenever one has $Du = 0$, as is the case for instance when the integrand $u$ and the fractional Brownian motion $W^H$ are independent stochastic processes, the relation \eqref{relationpathwise-divergence} says
\begin{equation*}
\int_0^t u_s dW^H_s =  \int_0^t u_s \delta W^H_s,
\end{equation*}
which is to say that the two approaches lead to the same integral and in particular that both are centered.
\end{remark}

\end{document}